\documentclass[12pt,reqno]{amsart}

\newif\ifarxiv
\arxivtrue

\ifarxiv
\usepackage[fontsize=13pt]{scrextend}

\usepackage[a4paper, margin=.75in]{geometry}
\else
\usepackage[a4paper, margin=.99in]{geometry}
\fi

\usepackage{amsmath, amsthm, thmtools, amsfonts, amssymb, mathtools}
\usepackage{pdflscape, blkarray, multirow, booktabs}
\usepackage[linesnumbered]{algorithm2e}
\usepackage[dvipsnames]{xcolor}
\usepackage{hyperref}
\hypersetup{
	colorlinks = true,
	linkcolor = {RoyalBlue},
	citecolor = {BrickRed},
	urlcolor = {Magenta}
}

\usepackage{float,import,xifthen,pdfpages,transparent}

\usepackage{tikz-cd}
\theoremstyle{plain}
\newtheorem{theorem}{Theorem}[section]
\newtheorem*{theorem*}{Theorem}

\newtheorem{prop}[theorem]{Proposition}
\newtheorem{corr}[theorem]{Corollary}
\newtheorem*{corr*}{Corollary}
\newtheorem*{problem*}{Problem}
\newtheorem{lemma}[theorem]{Lemma}

\theoremstyle{definition}
\newtheorem{defn}[theorem]{Definition}
\newtheorem{example}[theorem]{Example}

\newtheorem*{question*}{Question}

\newtheorem*{conj*}{Conjecture}
\newtheorem{remark}[theorem]{Remark}
\newtheorem*{obs*}{Observation}

\DeclareMathOperator{\codim}{codim}
\DeclareMathOperator{\rk}{rk}
\DeclareMathOperator{\cork}{cork}
\DeclareMathOperator{\im}{Im}

\DeclareMathOperator{\supp}{supp}

\DeclareMathOperator{\Op}{\mathcal{O}p}

\DeclareMathOperator{\Sol}{Sol}

\DeclareMathOperator{\Span}{Span}

\begin{document}
\allowdisplaybreaks
\title[$h$-Principle for Transverse Maps]{The $h$-Principle for Maps Transverse to Bracket-Generating Distributions}

\author[A. Bhowmick]{Aritra Bhowmick}
\address{Department of Mathematics and Statistics, IISER Kolkata \\ West Bengal, India}
\email{aritrabh.pdf@iiserkol.ac.in}



\begin{abstract}
Given a smooth bracket-generating distribution $\mathcal{D}$ of constant growth on a manifold $M$, we prove that maps from an arbitrary manifold $\Sigma$ to $M$, which are transverse to $\mathcal{D}$, satisfy the complete $h$-principle. This partially settles a question posed by M. Gromov.
\end{abstract}

\date{\today}
\maketitle

\setcounter{tocdepth}{3}

\frenchspacing

\section{Introduction}
A \emph{distribution} $\mathcal{D}$ on a manifold $M$ is a (smooth) subbundle of the tangent bundle $TM$. Given such a $\mathcal{D}$, we can consider the sheaf $\Gamma \mathcal{D}$ of local sections to $\mathcal{D}$, i.e, local vector fields on $M$ taking values in $\mathcal{D}$. The distribution $\mathcal{D}$ is called \emph{bracket-generating} if at each point $x \in M$, the tangent space $T_x M$ is spanned by the vector fields obtained by taking finitely many successive Lie brackets of vector fields in $\Gamma \mathcal{D}$. We say $\mathcal{D}$ is \emph{$(r-1)$-step bracket-generating}\footnote{Note that what we call a $(r-1)$-step bracket-generating is usually called an $r$-step bracket-generating distribution elsewhere in the literature.} at $x \in M$, if there exists some integer $r = r(x)$ such that \[T_x M = \Span \Big\{[X_1, \ldots [X_{k-1}, X_k]\cdots]_x \;\Big|\; X_1,\ldots,X_k \in \Gamma \mathcal{D}, \; 1 \le k \le r\Big\}.\]
Bracket-generating distributions are the stepping stone for the field of sub-Riemannian geometry \cite{gromovCCMetric, montTour}.

One possible way to study a given distribution $\mathcal{D}$ is via smooth maps $u: \Sigma \to M$ from an arbitrary manifold $\Sigma$ and looking at how the image of the differential $du: T\Sigma \to TM$ intersects $\mathcal{D}$. The map $u$ is said to be \emph{transverse} to $\mathcal{D}$ if $du(T_\sigma \Sigma) + \mathcal{D}_{u(\sigma)} = T_{u(\sigma)}M$ holds for every $\sigma \in \Sigma$. In \cite{gromovBook}, Gromov asked the reader to prove the following.
\begin{theorem*}
	Given a bracket-generating distribution $\mathcal{D}$ on a manifold $M$, maps $\Sigma \to M$ transverse to $\mathcal{D}$ satisfy the $h$-principle.
\end{theorem*}
The idea of proof as indicated in \cite[pg. 84]{gromovBook} contained an error that Gromov later acknowledged in \cite[pg. 254]{gromovCCMetric}. In \cite[pg. 136]{eliashbergBook}, the authors showed that the proof indeed goes through if the distribution $\mathcal{D}$ is contact. In fact, their argument remains valid for any strongly bracket-generating distribution (or fat distribution, see \cite{montTour} for a definition). Moreover, they also planned out a strategy that could work for an arbitrary bracket-generating distribution as well. In \cite{pinoMicroflexibleTransverse}, the authors carried out the ideas of \cite{eliashbergBook} and proved the $h$-principle for smooth maps transverse to \emph{real analytic} bracket-generating distributions on a real analytic manifold. Their argument heavily depends on estimating the codimension of certain semi-analytic sets in the jet bundle. It was also conjectured in the same article that Gromov's original statement should hold for a \emph{smooth} bracket-generating distribution if certain higher order jet calculations are performed.

The main goal of this article is to identify a suitable higher jet `regularity' condition (\autoref{defn:weaklyRegular}) so that the sheaf of $\mathcal{D}$-horizontal maps $\mathbb{R} \to M$ satisfying this regularity is microflexible. The difficulty lies in proving the local $h$-principle for this class of maps provided the distribution is equiregularly bracket-generating (\autoref{defn:equiregular}), which is proved in \autoref{thm:jetLifting}. Then, applying Gromov's analytic and sheaf-theoretic techniques, the $h$-principles for transverse maps (\autoref{thm:hPrinTrans}) and for transverse immersions (\autoref{thm:hPrinTransImm}), follow by a standard argument. We refer to \cite{gromovBook, eliashbergBook} for the details of this theory.

The article is organized as follows: In \autoref{sec:bracketGenDist}, we recall some basic notions about bracket-generating distributions. Then in \autoref{sec:weakRegularity}, we obtain the regularity criterion for maps $\mathbb{R} \to M$ horizontal to a bracket-generating distribution and prove the local $h$-principle for such maps. Lastly, in \autoref{sec:hPrinTrans} we prove the main $h$-principles.

\section{Bracket-Generating Distributions}\label{sec:bracketGenDist}
\begin{defn}
	A \emph{distribution} of rank $n$ (and corank $p$) on a manifold $M$ is a smooth vector subbundle of rank $n$ (and corank $p$) of the tangent bundle $TM$.
\end{defn}

Given any distribution $\mathcal{D} \subset TM$, we have the sheaf of local sections $\Gamma \mathcal{D}$, which is a sheaf of local vector fields on $M$. By the notation $X \in \mathcal{D}$ we shall mean a local section $X \in \Gamma \mathcal{D}$ defined on some unspecified open set of $M$. Given two arbitrary sheaves $\mathcal{E},\mathcal{F}$ of vector fields on $M$ (not necessarily given as sheaves of sections of some distribution), we can define the sheaf 
\[ [ \mathcal{E}, \mathcal{F} ] = \Span \Big\{ [X,Y] \; \Big| \; \text{$X \in \mathcal{E}, Y \in \mathcal{F}$} \Big\},\]
where the span is taken over $C^\infty(M)$. Inductively, we then define \[\mathcal{D}^0 = 0, \qquad \mathcal{D}^1 = \mathcal{D}, \qquad \mathcal{D}^{i+1} = \mathcal{D}^i + [\mathcal{D},\mathcal{D}^i],\; i \ge 1.\]

\begin{defn}
	A distribution $\mathcal{D}$ is said to be \emph{bracket-generating} if for each $x \in M$ we have $T_x M = \mathcal{D}^{r+1}_x$ for some $r = r(x)$. $\mathcal{D}$ is said to have type $\mathfrak{m} = \mathfrak{m}(x)$ at $x$ for the $(r+2)$-tuple $\mathfrak{m} = \big(0=m_0 \le \cdots \le m_{r+1} = \dim M\big)$, where $m_i = \rk \mathcal{D}^i_x$ for $0 \le i \le r+1$.
\end{defn}

For a \emph{generic} bracket-generating distribution $\mathcal{D}$ on $M$, the number of steps it takes to bracket-generate $T_x M$ is non-constant and the sheaves $\mathcal{D}^i$ may fail to be of constant rank.

\begin{defn}\label{defn:equiregular}
	A bracket-generating distribution $\mathcal{D}$ on $M$ is said to be $r$-step bracket-generating if $T_x M = \mathcal{D}^{r+1}_x$ for all $x \in M$. Furthermore, $\mathcal{D}$ is said to be \emph{equiregular} (or, \emph{of constant growth}) of type $\mathfrak{m} = (m_0 < \cdots < m_{r+1})$ if $\mathcal{D}$ has the type $\mathfrak{m}$ at every $x\in M$.
\end{defn}

Throughout this article, we shall mostly restrict ourselves to equiregular distributions $\mathcal{D}$ of some fixed type $\mathfrak{m}$. In particular, each $\mathcal{D}^s$ will be a distribution, and we get a flag 
\[0=\mathcal{D}^0 \subset \mathcal{D}=\mathcal{D}^1 \subset \mathcal{D}^2 \subset \dots \subset \mathcal{D}^{r+1} = TM.\]
It should be noted that in general, equiregularity is a non-generic condition on the germs of distributions of a given rank, although most of the interesting distributions appearing in the literature possess this property.

\begin{example}
	Contact and Engel distributions are well-studied examples of bracket-generating distributions that bracket-generates the tangent bundle in $1$ and $2$ steps respectively. More generally, we have Goursat structures which are certain rank $2$, $r$-step bracket-generating distributions on manifolds of dimension $r + 2$. Note that all of these distributions are equiregular as well. On the other hand, the Martinet distribution, given as the kernel $\ker\left( dz - y^2 dx \right)$ on $\mathbb{R}^3$ is \emph{not} equiregular. We refer to \cite{montTour} for many more examples.
\end{example}

We shall need the following lemma in the next section.

\begin{lemma}\label{lemma:specialVectorFields}
	Let $\mathcal{D}$ be an equiregular bracket-generating distribution on $M$, of type $\mathfrak{m}=(m_0 < \cdots < m_{r+1})$. Set $p_s = \rk \big(\mathcal{D}^{s+1}/\mathcal{D}^{s}) = m_{s+1} - m_s > 0$. Then, for any $x \in M$ and for $1\le j \le p_s, 1\le s\le r$, there exists a collection of vectors \[\tau^{s,j} \in \mathcal{D}_x, \quad \eta^{s,j} \in \mathcal{D}^{s}_x \setminus \mathcal{D}^{s-1}_x, \quad \zeta^{s,j} \in \mathcal{D}^{s+1}_x \setminus \mathcal{D}^s_x,\]
	and $1$-forms $\lambda^{s,j}$ defined near $x$, such that the following holds.
	\begin{itemize}
		\item For each $1 \le s \le r$, $\mathcal{D}_x^{s+1} = \mathcal{D}^s_x + \Span \langle \zeta^{s,1}_x,\ldots,\zeta^{s,p_s}_x\rangle$.
		
		\item The $1$-forms $\{\lambda^{s,j}\}$ are dual to $\{\zeta^{s,j}\}$ at $x$. Furthermore, $\{\lambda^{s,j} \; | \; 1\le j \le p_s, 1\le s \le r\}$ is a frame for the annihilator bundle of $\mathcal{D}$ near $x$.
		
		\item For each $1 \le s \le s^\prime \le r$ and $1 \le j \le p_s, 1\le j^\prime \le p_{s^\prime}$, we have
		\[d\lambda^{s^\prime, j^\prime}|_x (\tau^{s,j}, \eta^{s,j}) = \begin{cases}
			\delta_{j,j^\prime}, \quad s^\prime = s \\
			0, \quad s^\prime > s,
		\end{cases}\]
		where $\delta_{j,j^\prime}$ is the Kronecker's delta function.
	\end{itemize}
\end{lemma}
\begin{proof}
	Since $\mathcal{D}^{s+1} = \mathcal{D}^s + [\mathcal{D}, \mathcal{D}^s]$, we have the well-defined sheaf homomorphism
	\begin{align*}
		\Omega^s : \mathcal{D} \otimes \mathcal{D}^s/ \mathcal{D}^{s-1} &\longrightarrow \mathcal{D}^{s+1} / \mathcal{D}^s \\
		X \otimes (Y \mod \mathcal{D}^{s-1}) &\longmapsto -[X,Y] \mod \mathcal{D}^s
	\end{align*}
	Furthermore, $\Omega^s$ is $C^\infty(M)$-linear and hence, for vectors $X \in \mathcal{D}_x, Y \in \mathcal{D}^s_x$ we have the linear maps
	\[\Omega^s_x(X, Y \mod \mathcal{D}^{s-1}_x) = -[\tilde{X}, \tilde{Y}]_x \mod \mathcal{D}^s_x,\]
	where $\tilde{X} \in \mathcal{D}, \tilde{Y} \in \mathcal{D}^s$ are arbitrary extensions of $X,Y$ respectively. Thus, $\Omega^s : \mathcal{D} \otimes \mathcal{D}^s/\mathcal{D}^{s-1} \rightarrow \mathcal{D}^{s+1}/\mathcal{D}^s$ is a bundle map, which is surjective since $\mathcal{D}$ is bracket-generating, for $1\le s\le r$.
	
	Choose vectors $\tau^{s,j} \in \mathcal{D}_x, \eta^{s,j} \in \mathcal{D}^{s}_x \setminus \mathcal{D}^{s-1}_x$ so that $\{ \Omega^s_x(\tau^{s,j}, \eta^{s,j}) \; | \; 1\le j \le p_s \}$ forms a frame of $\mathcal{D}^{s+1}_x/\mathcal{D}^s_x$. Let us consider some arbitrary extensions $\tilde{\tau}^{s,j} \in \mathcal{D}, \tilde{\eta}^{s,j} \in \mathcal{D}^s \setminus \mathcal{D}^{s-1}$ of $\tau^{s,j}, \eta^{s,j}$ respectively, and denote $\tilde{\zeta}^{s,j} = -[\tilde{\tau}^{s,j}, \tilde{\eta}^{s,j}] \in \mathcal{D}^{s+1}$. Note that $\Omega^s(\tilde{\tau}^{s,j}, \tilde{\eta}^{s,j}) = \tilde{\zeta}^{s,j} \mod \mathcal{D}^s$. Since $TM = \oplus_{s=0}^r \mathcal{D}^{s+1}/\mathcal{D}^s$, we have a local framing $TM \underset{\textrm{loc.}}{=}  \mathcal{D} \oplus \Span \langle \tilde{\zeta}^{s,j}, 1\le j \le p_s, 1\le s \le r \rangle$ near $x$. Next, choose independent local $1$-forms $\lambda^{s,j}$ near $x$, which are in the annihilator bundle $\textrm{Ann} \mathcal{D}$ (i.e., $\lambda^{s,j}|_{\mathcal{D}} = 0$), and $\left\{ \lambda^{s,j} \right\}$ is dual to $\left\{ \tilde{\zeta}^{s,j} \right\}$. Note that 
	\[\mathcal{D}^s \underset{\textrm{loc.}}{=} \bigcap_{s^\prime \ge s} \bigcap_{j=1}^{p^{s^\prime}} \ker \lambda^{s^\prime,j}, \quad 1 \le s \le r.\]
	Hence, for $s^\prime \ge s$, we have
	\begin{align*}
		d\lambda^{s^\prime, j^\prime}(\tau^{s,j}, \eta^{s,j}) &= \left[ \tilde{\tau}^{s,j} \big( \underbrace{\lambda^{s^\prime, j^\prime}\left( \tilde{\eta}^{s,j} \right)}_{0} \big) - \tilde{\eta}^{s,j} \big( \underbrace{\lambda^{s^\prime, j^\prime}\left( \tilde{\tau}^{s,j} \right)}_{0} \big) - \lambda^{s^\prime, j^\prime}\big( \underbrace{[\tilde{\tau}^{s,j}, \tilde{\eta}^{s,j}]}_{-\tilde{\zeta}^{s,j}} \big) \right]_x \\
		&= \lambda^{s^\prime, j^\prime}(\zeta^{s,j}) \\
		&= \begin{cases}
			\delta_{j^\prime,j}, \; s^\prime = s \\
			0, \; s^\prime > s.
		\end{cases}
	\end{align*}
	This concludes the proof.	
\end{proof}

\section{Regularity of Horizontal Curves} \label{sec:weakRegularity}
Let us fix an arbitrary distribution $\mathcal{D}$ with rank $n$ and corank $p$ on $M$ with $\dim M = N = n + p$. Given a manifold $\Sigma$, a map $u : \Sigma \to M$ is called \emph{$\mathcal{D}$-horizontal} if $du(T_\sigma \Sigma) \subset \mathcal{D}_{u(\sigma)}$ for each $\sigma \in \Sigma$. For simplicity, let us assume that $\mathcal{D}$ is given as the kernel of $1$-forms $\lambda^1, \ldots, \lambda^p$ on $M$. Then, $\mathcal{D}$-horizontal maps $\Sigma \to M$ are precisely the solutions of the following \emph{non-linear} differential operator:
\begin{equation}
	\label{eqn:opD}
	\begin{aligned}
		\mathfrak{D} : C^\infty(\Sigma, M) &\longrightarrow \Omega^1(\Sigma, \mathbb{R}^p) = \Gamma \hom(T\Sigma, \mathbb{R}^p) \\
		u &\longmapsto \big(u^*\lambda^1, \ldots, u^*\lambda^p\big)
	\end{aligned}
\end{equation}
To find solutions of $\mathfrak{D}$, we appeal to the Nash-Gromov implicit function theorem. As a first step, linearizing $\mathfrak{D}$ at some $u : \Sigma \to M$ we get the \emph{linear} differential operator:
\begin{equation}
	\label{eqn:opL_u}
	\begin{aligned}
		\mathfrak{L}_u : \Gamma u^*TM &\longrightarrow \Omega^1(\Sigma, \mathbb{R}^p)\\
		\xi &\longmapsto \Big[X \mapsto \Big(d\lambda^s(\xi, u_*X) + X\big(\lambda^s \circ \xi\big)\Big)_{s=1}^p\Big],
	\end{aligned}
\end{equation}
which restricts to the \emph{bundle map} on $\Gamma u^*\mathcal{D}$:
\[\mathcal{L}_u \coloneqq \mathfrak{L}_u|_{\Gamma u^*\mathcal{D}} : \xi \longmapsto \Big[X \mapsto \Big( d\lambda^s(\xi, u_*X)\Big) \Big].\]
An immersion $u : \Sigma \to M$ is called \emph{$(d\lambda^s)$-regular} if the bundle map $\mathcal{L}_u$ is surjective. In general, $(d\lambda^s)$-regularity depends on our choice of $1$-forms $\lambda^s$, whereas $(d\lambda^s)$-regularity of a $\mathcal{D}$-horizontal map is independent of any such choice. A $(d\lambda^s)$-regular horizontal immersion is also called \emph{$\Omega$-regular}, where $\Omega : \Lambda^2 \mathcal{D} \rightarrow  TM/\mathcal{D}$ is the associated curvature $2$-form. It follows that the sheaf of $\Omega$-regular horizontal maps $\Sigma \to M$ is microflexible \cite[pg. 339]{gromovBook}. Note that $(d\lambda^s)$-regularity is a first order condition on the class of maps. For the existence of a $(d\lambda^s)$-regular horizontal map, even when $\dim \Sigma = 1$, $\mathcal{D}$ must be $1$-step bracket-generating. We now identify a suitable higher order regularity for maps $\mathbb{R} \to M$ horizontal to an arbitrary distribution.

\subsection{$\mathcal{W}$-Regular Horizontal Curves}
The arguments presented in this section follow the general scheme of algebraically solving (underdetermined) linear partial differential operators as in \cite[pg. 155]{gromovBook}. Briefly, the idea is as follows. In order to solve the linear operator $\mathfrak{L}_u$ for some $u$, ideally we need to find out some linear operator $\mathfrak{M}_u : \Omega^1(\Sigma, \mathbb{R}^p) \rightarrow \Gamma u^*TM$ so that $\mathfrak{L}_u \circ \mathfrak{M}_u = \mathrm{Id}$ holds. This involves solving partial differential equations in the coefficients of $\mathfrak{M}_u$, which is in general hard to do. Instead, we try to look for a linear operator $\mathfrak{S}_u : \Gamma u^*TM \rightarrow \Omega^1(\Sigma, \mathbb{R}^p)$ satisfying $\mathfrak{S}_u \circ \overline{\mathfrak{L}_u} = \mathrm{Id}$, where $\overline{\mathfrak{L}_u} : \Omega^1(\Sigma, \mathbb{R}^p) \rightarrow \Gamma u^*TM$ is the  formal adjoint of $\mathfrak{L}_u$. Note that this system is \emph{algebraic} in the coefficients of $\mathfrak{S}_u$, and thus are considerably easier to solve. Once we get a smooth solution $\mathfrak{S}_u$, we can take the formal adjoint of the whole equation, and obtain $\mathfrak{L}_u \circ \overline{\mathfrak{S}_u} = \mathrm{Id}$, since $\overline{\overline{\mathfrak{L}_u}} = \mathfrak{L}_u$. Taking $\mathfrak{M}_u = \overline{\mathfrak{S}_u}$ we then have the desired solution to the original problem. This observation was used by Gromov to prove the fact that a generic \emph{underdetermined} linear operator admits (universal) a right inverse \cite[pg. 156]{gromovBook}. Although $\mathfrak{L}_u$ is underdetermined, we are not able to appeal to this theorem directly, since we do not know whether the operator is sufficiently generic in the sense of Gromov. Instead, we explicitly identify a class of maps $u$ for which the algebraic system always admits a (smooth) solution. We would like to note that a similar approach was also successfully used in \cite{leoNonFree}, where the author proved the existence of non-free isometric immersions.

Without loss of generality, we assume $M = \mathbb{R}^N$ and fix some coordinates $y^1,\ldots, y^N$ on $M$. Let us write the $1$-forms $\lambda^s$ as \[\lambda^s = \lambda^s_\mu dy^\mu, \quad 1\le s\le p.\] We also fix the (global) coordinate $t$ on $\mathbb{R}$. For a function $u : \mathbb{R} \to \mathbb{R}^N$, the linearization operator $\mathfrak{L}_u$ (\autoref{eqn:opL_u}) is then given as
\begin{equation}\label{eqn:L_uFormula}
	\begin{aligned}
		\mathfrak{L}_u : C^\infty(\mathbb{R}, \mathbb{R}^N) &\longrightarrow C^\infty(\mathbb{R}, \mathbb{R}^p)\\
		\xi = \big(\xi^\mu\big) &\longrightarrow \Big(\big(\partial_\mu \lambda^s_\nu \circ u\big) \partial_t u^\nu \xi^\mu + \big(\lambda^s_\mu \circ u\big) \partial_t \xi^\mu\Big)_{s=1}^p
	\end{aligned}
\end{equation}
Written in a matrix form we have $\mathfrak{L}_u(\xi) = \mathfrak{L}_u^0 \xi + \mathfrak{L}_u^1 \partial_t \xi$ where the $p \times N$ matrices $\mathfrak{L}_u^i$ are given as
\begin{equation}\label{eqn:L_uMatrixFormula}
	\mathfrak{L}_u^0 = \begin{pmatrix}
		\big(\partial_\mu \lambda^s_\nu \circ u\big)\partial_t u^\nu
	\end{pmatrix}_{p \times N}, \quad \mathfrak{L}_u^1 = \begin{pmatrix}
		\lambda^s_\mu \circ u
	\end{pmatrix}_{p \times N}.
\end{equation}
Taking the formal adjoint of $\mathfrak{L}_u$, we get the first order operator $\mathfrak{R}_u : C^\infty(\mathbb{R}, \mathbb{R}^p) \to C^\infty(\mathbb{R}, \mathbb{R}^N)$ as follows
\begin{equation}\label{eqn:R_uAdjointFormula}
	\mathfrak{R}_u = \mathfrak{R}_u^0 + \mathfrak{R}_u^1 \partial_t = \big(\mathfrak{L}_u^0 - \partial_t \mathfrak{L}_u^1\big)^\dagger - \big(\mathfrak{L}_u^1\big)^\dagger \partial_t.
\end{equation}
Observe that
\begin{equation}
	\left.\begin{aligned}
		\mathfrak{L}_u^0 - \partial_t \mathfrak{L}_u^1
		&= \Big(\big(\partial_\mu \lambda^s_\nu \circ u\big)\partial_t u^\nu - \partial_t \big(\lambda^s_\mu \circ u\big) \Big)_{p\times N} \\
		&= \Big( \big(\partial_\mu \lambda^s_\nu \circ u - \partial_\nu \lambda^s_\mu \circ u\big)\partial_t u^\nu \Big)_{p\times N} \\
		&= \Big( -(\iota_{u_*\partial_t}d\lambda^s)(\partial_\mu) \Big)_{p\times N}.
	\end{aligned}\right.
\end{equation}
Let us now consider the equation \begin{equation}
	\mathfrak{S} \circ \mathfrak{R}_u = \mathrm{Id}, \label{eqn:SR=Id}
\end{equation}
for an arbitrary order $q$ linear operator $\mathfrak{S} : C^\infty(\mathbb{R}, \mathbb{R}^N) \to C^\infty(\mathbb{R}, \mathbb{R}^p)$ given as \[\mathfrak{S} \coloneqq \mathfrak{S}^0 + \mathfrak{S}^1 \partial_t + \cdots + \mathfrak{S}^q \partial_t^q,\]
where $\mathfrak{S}^i$ are $p \times N$ matrices of functions. Note that \autoref{eqn:SR=Id} is \emph{algebraic} in the entries of the matrices $\mathfrak{S}^i$. In fact, \autoref{eqn:SR=Id} represents a total of $p^2(q+2)$ many equations in $pN(q+1)$ many variables, namely, $\big\{\mathfrak{S}^i_{\alpha\beta} \;\big|\; 1 \le \alpha \le p, \; 1 \le \beta \le N, \; 0 \le i \le q\big\}$. This system is \emph{underdetermined} if and only if 
\begin{equation}\label{eqn:underdetermined}
	p^2(q+2) < pN(q+1) \Leftrightarrow n q > p - n.
\end{equation}
Expanding out \autoref{eqn:SR=Id} we have
\begin{align*}
	\mathrm{Id} &= \mathfrak{S} \circ \mathfrak{R}_u\\
	&= \big(\mathfrak{S}^0 + \mathfrak{S}^1\partial_t + \cdots + \mathfrak{S}^q \partial_t^q\big) \circ \big(\mathfrak{R}_u^0 + \mathfrak{R}_u^1 \partial_t\big)\\
	&= \Big(\mathfrak{S}^0 \mathfrak{R}_u^0 + \mathfrak{S}^1 \partial_t \mathfrak{R}_u^0 + \cdots + \mathfrak{S}^q \partial_t^q \mathfrak{R}_u^0\Big) \\
	&\qquad + \Big(\mathfrak{S}^0 \mathfrak{R}_u^1 + \mathfrak{S}^1 \big(\mathfrak{R}_u^0 + \partial_t \mathfrak{R}_u^1\big) + \cdots \mathfrak{S}^q \big(q\partial_t^{q-1} \mathfrak{R}_u^0 + \partial_t^q \mathfrak{R}_u^1\big)\Big) \partial_t\\
	&\qquad + \cdots \\
	&\qquad + \Big(\mathfrak{S}^{q-1}\mathfrak{R}^1_u + \mathfrak{S}^q\big(\mathfrak{R}_u^0 + q\partial_t\mathfrak{R}_u^1\big)\Big)  \partial_t^q\\
	&\qquad + \mathfrak{S}^q \mathfrak{R}_u^1 \partial_t^{q+1}
\end{align*} 
Comparing both sides, we get the block-matrix system:
\begin{equation} \label{eqn:SR=Id_Matrix}
	\resizebox{\linewidth}{!}{$\begin{aligned}
		\begin{pmatrix}
			\mathfrak{S}^0 & \mathfrak{S}^1 & \cdots & \mathfrak{S}^{q-1} & \mathfrak{S}^q
		\end{pmatrix}
		\begin{pmatrix}
			\mathfrak{R}_u^0 & \mathfrak{R}_u^1 & \cdots & 0 & 0\\
			\partial_t \mathfrak{R}_u^0 & \mathfrak{R}_u^0 + \partial_t \mathfrak{R}_u^1 & \cdots & 0 & 0\\
			\vdots & \vdots & \vdots & \vdots \\
			\partial_t^{q-1} \mathfrak{R}_u^0 & (q-1)\partial_t^{q-2}\mathfrak{R}_u^0 + \partial_t^{q-1}\mathfrak{R}_u^1 & \cdots & \mathfrak{R}_u^1 & 0\\
			\partial_t^q \mathfrak{R}_u^0 & q\partial_t^{q-1} \mathfrak{R}_u^0 + \partial_t^q \mathfrak{R}_u^1 & \cdots & \mathfrak{R}_u^0 + q\partial_t \mathfrak{R}_u^1 & \mathfrak{R}_u^1
		\end{pmatrix} \\
		=
		\begin{pmatrix}
			\mathrm{Id}_{p\times p} &
			0_{p\times p} &
			\cdots &
			0_{p\times p} &
			0_{p\times p}
		\end{pmatrix}
	\end{aligned}$}
\end{equation}
Let us denote
\begin{equation}\label{eqn:R_uLambdaNotation}
	R_u \coloneqq - \left( \mathfrak{L}_u^0 - \partial_t \mathfrak{L}_u^1 \right) = \begin{pmatrix}
		\iota_{u_*\partial_t} d\lambda^s
	\end{pmatrix}_{p\times N}, \quad \Lambda \coloneqq \mathfrak{L}_u^1 = \begin{pmatrix}
		\lambda^s_\mu \circ u
	\end{pmatrix}_{p\times N},
\end{equation}
so that we have from \autoref{eqn:R_uAdjointFormula}
\begin{equation}
	\mathfrak{R}_u^0 = -\left( R_u \right)^\dagger \quad \text{and} \quad \mathfrak{R}_u^1 = -\Lambda^\dagger.
\end{equation} 
Taking the adjoint of the coefficient matrix in \autoref{eqn:SR=Id_Matrix} and multiplying by $-1$ we then get the following matrix.
\begin{equation}\label{matrix:original}
	A \coloneqq \begin{pmatrix}
		R_u & \partial_t R_u & \cdots & \partial_t^{q-1}R_u & \partial_t^q R_u \\
		\Lambda & R_u + \partial_t \Lambda &  \cdots & (q-1)\partial_t^{q-1}R_u + \partial_t^{q-1}\Lambda & q\partial_t^{q-1}R_u + \partial_t^q \Lambda \\
		0 & \Lambda & \cdots & \cdots & \cdots \\
		\vdots & \vdots & \ddots & \vdots & \vdots \\
		0 & 0 & \cdots & \Lambda & R_u + q\partial_t\Lambda \\
		0 & 0 & \cdots & 0 & \Lambda \\
	\end{pmatrix}.
\end{equation}
$A$ is a matrix of size $p(q+2) \times N(q+1)$, which depends on the $(q+1)^{\text{th}}$-jet of the map $u$. The rows of this matrix can be linearly independent only if $nq \ge p - n$ holds (\autoref{eqn:underdetermined}). Under the full rank condition, one can always solve for $\mathfrak{S}^i$ smoothly in \autoref{eqn:SR=Id_Matrix}, and thus solving \autoref{eqn:SR=Id}. It should be noted that there is no unique solution, but instead, we have an affine space of them. The full rank of $A$ will enable us to choose a solution that varies smoothly depending on $j^{q+1}_u(x)$.

\begin{defn}\label{defn:weaklyRegular}
	For some fixed $q$ satisfying $n q \ge p-n$, define the relation $\mathcal{W} \subset J^{q+1}(\mathbb{R}, \mathbb{R}^N)$:
	\[\mathcal{W} = \Big\{j^{q+1}_u(x) \;\Big|\; \text{$du_x$ is injective and the matrix $A = A\big(j^{q+1}_u(x)\big)$ has full (row) rank} \Big\}.\]
	A smooth solution of $\mathcal{W}$ is called a \emph{$\mathcal{W}$-regular} or \emph{weakly $(d\lambda^s)$-regular} map. We denote by $\Sol \mathcal{W}$ the space of all $\mathcal{W}$-regular maps.
\end{defn}

\begin{lemma}
	$(d\lambda^s)$-regular maps are $\mathcal{W}$-regular.
\end{lemma}
\begin{proof}
	If an immersion $u: \mathbb{R} \to \mathbb{R}^N$ is $(d\lambda^s)$-regular, then the block $\begin{pmatrix} R_u \\ \Lambda \end{pmatrix}_{2p \times N}$ in the top-left corner of $A$ has full (row) rank, which makes the first two row-blocks of $A$ full rank. On the other hand, the $\Lambda$ blocks on the `off-diagonal' are always full rank, since the rows of $\Lambda$ consist of linearly independent $1$-forms $\{\lambda^s\}$. Note that there is no overlap between the $\begin{pmatrix} R_u \\ \Lambda \end{pmatrix}_{2p \times N}$ block and the rest of the diagonal $\Lambda$ blocks. Hence, the rows of $A$ are linearly independent whenever $u$ is $(d\lambda^s)$-regular, i.e., $u$ is then $\mathcal{W}$-regular.
\end{proof}

In light of the above lemma, one observes that the first row-block of $A$ is the one where the rank of $A$ might drop, and one may consider $\mathcal{W}$-regularity as the natural higher order analog of $d\lambda^s$-regularity. This observation shall become more clear in the proof of \autoref{thm:jetLifting}. Let us now show that $\mathfrak{L}_u$ admits a universal right inverse over $\mathcal{W}$-regular maps, i.e., one can solve $\mathfrak{L}_u \circ \mathfrak{M}_u = \mathrm{Id}$ for any $\mathcal{W}$-regular map $u$, such that $\mathfrak{M}_u$ depends smoothly on $u$.

\begin{prop}\label{prop:inversion}
	Fix $q$ satisfying $nq \ge p - n$ and the relation $\mathcal{W} \subset J^{q+1}(\mathbb{R}, \mathbb{R}^N)$. Then, for $\mathcal{W}$-regular maps $u : \mathbb{R} \to \mathbb{R}^N$, there exists a linear partial differential operator $\mathfrak{M}_u : C^\infty(\mathbb{R}, \mathbb{R}^p) \to C^\infty(\mathbb{R}, \mathbb{R}^N)$ of order $q$, satisfying $\mathfrak{L}_u \circ \mathfrak{M}_u = \mathrm{Id}$. Furthermore, the assignment \[\Sol \mathcal{W}\times C^\infty(\mathbb{R}, \mathbb{R}^p) \ni (u,P) \mapsto \mathfrak{M}(u,P) \coloneqq \mathfrak{M}_u(P) \in C^\infty(\mathbb{R}, \mathbb{R}^N)\] is a differential operator, nonlinear of order $2q+1$ in the first variable.
\end{prop}
\begin{proof}
	Fix a jet $\sigma = j^{q+1}_u(t) \in \mathcal{W}|_t$, represented by some map $u : \Op(t) \to \mathbb{R}^N$. The first order operator $\mathfrak{R}_u$ defined on $\Op(t)$ gives rise to the (linear) symbol map \[\Delta_{\mathfrak{R}_\sigma} : J^{1}(\mathbb{R}, \mathbb{R}^p)|_t \to J^0(\mathbb{R},\mathbb{R}^N)|_t = C^\infty(\mathbb{R}, \mathbb{R}^N)|_t.\]
	For any jet $j^1_P(t) \in J^1(\mathbb{R}, \mathbb{R}^p)|_t$ represented by some $P : \Op(t) \to \mathbb{R}^p$, we then have $\Delta_{\mathfrak{R}_\sigma}(j^1_P(t) ) = \mathfrak{R}_u (P)(t)$. We define $\Delta^{(q)}_{\mathfrak{R}_\sigma} : J^{q+1}(\mathbb{R},\mathbb{R}^p)|_t \to J^q(\mathbb{R}, \mathbb{R}^N)|_t$ by \[\Delta^{(q)}_{\mathfrak{R}_\sigma} \big(j^{q+1}_P(t)\big) = j^{q}_{\mathfrak{R}_u(P)}(t).\] Since $\sigma \in \mathcal{W}$, the matrix $A_\sigma = A(j^{q+1}_u(t))$ has full row rank. We can then readily solve for $\mathfrak{S}^j = \mathfrak{S}^j(\sigma)$ in \autoref{eqn:SR=Id_Matrix}, in terms of rational polynomials of the terms of $A_\sigma$. Indeed, we get a (linear) map $\Delta_{\mathfrak{S}_\sigma} : J^{q}(\mathbb{R},\mathbb{R}^N)|_t \to C^\infty(\mathbb{R}, \mathbb{R}^p)|_t$ satisfying the commutative diagram
	\[\begin{tikzcd}
		J^{q+1}(\mathbb{R},\mathbb{R}^p)|_t \arrow{rr}{\Delta^{(q)}_{\mathfrak{R}_\sigma}} \arrow{rd}[swap]{p^{q+1}_0} && J^q(\mathbb{R},\mathbb{R}^N)|_t \arrow{ld}{\Delta_{\mathfrak{S}_\sigma}} \\
		& J^0(\mathbb{R},\mathbb{R}^p)|_t
	\end{tikzcd}\]
	Consider an open neighborhood $U(\sigma) \subset \mathcal{W}$ of $\sigma$ so that the denominators of all the rational polynomials in $\Delta_{\mathfrak{S}_\sigma}$ remain nonzero for all jets $\tau \in U(\sigma)$. Shrinking $U(\sigma)$ if necessary, assume that $U(\sigma)$ projects down to an open neighborhood $V(\sigma)\subset \mathbb{R}$ of $t$. We then have a smooth map \[\Delta_\sigma : U(\sigma) \times J^{q}(V(\sigma), \mathbb{R}^N) \to C^\infty(V(\sigma), \mathbb{R}^p),\]
	so that for $\tau \in U(\sigma)|_s$ with $s\in V(\sigma)$ and for any $j^{q+1}_P(s)$, the following holds: \[ \Delta_\sigma \Big(\tau, \; \Delta_{\mathfrak{R}_\tau}^{(q)}\big(j^{q+1}_P(s)\big)\Big) = p^{q+1}_0\big(j^{q+1}_P(s)\big) = P(s).\]
	Note that $\Delta_\sigma$ is nonlinear in the first term, whereas it is linear in the second term.
	
	We now have an open cover $\mathfrak{U} = \{U(\sigma)\}_{\sigma\in\mathcal{W}}$ of $\mathcal{W}$. Fix a partition of unity $\{\rho_\alpha\}_{\alpha \in \Lambda}$ on $\mathcal{W}$ subordinate to $\mathfrak{U}$, so that $\supp \rho_\alpha \subset U_\alpha$ for some $U_\alpha \in \mathfrak{U}$. We denote the corresponding open set $V_\alpha \subset \mathbb{R}$ and the map $\Delta_\alpha : U_\alpha \times J^{q}(V_\alpha, \mathbb{R}^N) \to C^\infty(V_\alpha, \mathbb{R}^p)$. Define the bundle map \[\Delta_\mathfrak{S} : \mathcal{W} \times J^{q}(\mathbb{R}, \mathbb{R}^N) \to C^\infty(\mathbb{R},\mathbb{R}^p)\]
	via the formula
	\[\Delta_\mathfrak{S}\big(\tau, \eta\big) \coloneqq \sum_\alpha \Delta_{\alpha}\big(\tau, \rho_\alpha(\tau) \eta \big).\]
	Since each $\Delta_\alpha$ is linear in the second argument, the map $\Delta_\mathfrak{S}$ is well-defined and smooth. Now, for jets $\tau = j^{q+1}_u(s) \in \mathcal{W}$ and $\eta = j^{q+1}_P(s) \in J^{q+1}(\mathbb{R},\mathbb{R}^p)$, we have
	\begin{align*}
		\Delta_\mathfrak{S}\big(\tau, \; \Delta^{(q)}_{\mathfrak{R}_\tau}(\eta)\big) &= \sum_\alpha \Delta_\alpha\big(\tau, \; \rho_\alpha(\tau) \Delta^{(q)}_{\mathfrak{R}_\tau}(\eta)\big) \\
		&= \sum_\alpha \Delta_\alpha\Big(\tau, \; \Delta^{(q)}_{\mathfrak{R}_\tau}\big(\rho_\alpha(\tau) \eta\big)\Big), \quad \text{as $\Delta^{(q)}_{\mathfrak{R}_\tau}$ is linear} \\
		&= \sum_\alpha p^{q+1}_0\big(\rho_\alpha(\tau) \eta \big), \quad \text{by the construction of $\Delta_\alpha$} \\
		&= p^{q+1}_0\Big(\sum_\alpha \rho_\alpha(\tau) \eta \Big) \\
		& = p^{q+1}_0(\eta) \\
		&= P(s)
	\end{align*}
	Define $\mathfrak{S} : \Sol\mathcal{W} \times C^\infty(\mathbb{R},\mathbb{R}^N) \to C^\infty(\mathbb{R}, \mathbb{R}^p)$ via the formula \[\mathfrak{S}(u, \xi) = \Delta_\mathfrak{S}\big(j^{q+1}_u, \; j^{q}_\xi\big).\]
	The operator $\mathfrak{S}$ is nonlinear of order $q+1$ in the first component and linear of order $q$ in the second component. We have $\mathfrak{S}\big(u, \mathfrak{R}_u(P)\big) = P$ for any $u\in \Sol\mathcal{W}$ and $P \in C^\infty(\mathbb{R}, \mathbb{R}^p)$.
	
	Lastly, define the operator $\mathfrak{M} : \Sol\mathcal{W} \times C^\infty(\mathbb{R}, \mathbb{R}^p) \to C^\infty(\mathbb{R}, \mathbb{R}^N)$ by \[\mathfrak{M}(u,P) = \mathfrak{M}_u(P) = \overline{\mathfrak{S}_u}(P),\]
	where $\overline{\mathfrak{S}_u} : C^\infty(\mathbb{R},\mathbb{R}^p) \to C^\infty(\mathbb{R}, \mathbb{R}^N)$ is the formal adjoint to the operator $\mathfrak{S}_u : \xi \mapsto \mathfrak{S}(u,\xi)$. We have \[\mathfrak{L}_u \circ \mathfrak{M}_u = \overline{\mathfrak{R}_u} \circ \overline{\mathfrak{S}_u} = \overline{\mathfrak{S}_u \circ \mathfrak{R}_u} = \overline{\mathrm{Id}} = \mathrm{Id}, \quad \text{for any $u\in \Sol \mathcal{W}$.}\] Clearly $\mathfrak{M}$ is a differential operator, which is linear of order $q$ in the second component. Since taking adjoint of the $q^\text{th}$-order operator $\mathfrak{S}_u$ itself has order $q$, we have $\mathfrak{M}$ is nonlinear of order $2q+1$ in the first component.
\end{proof}

Following Gromov's terminology \cite[pg. 115-116]{gromovBook}, \autoref{prop:inversion} implies that for any $q$ satisfying $nq \ge p - n$, the first order operator $\mathfrak{D} : C^\infty(\mathbb{R}, \mathbb{R}^N) \to C^\infty(\mathbb{R}, \mathbb{R}^p)$ is infinitesimally invertible over $\mathcal{W}$-regular maps, with defect $2q+1$ and order of inversion $q$. For $\alpha \ge 0$, denote the relation of $\alpha$-infinitesimal solutions of $\mathfrak{D} = 0$ as 
\begin{equation}
	\label{eqn:rel:tangentAlpha}
	\mathcal{R}^\alpha_{\textrm{tang}} = \Big\{j^{\alpha+1}_u(x) \; \Big| \; j^{\alpha}_{\mathfrak{D}(u)}(x) = 0\Big\} \subset J^{\alpha + 1}(\mathbb{R}, \mathbb{R}^N), \quad \alpha \ge 0.
\end{equation}
Next, for $\alpha \ge (2q + 1) - 1 = 2q$ denote the relation of $\mathcal{W}$-regular $\alpha$-infinitesimal solutions of $\mathfrak{D} = 0$ by
\begin{equation}
	\label{eqn:rel:calWAlpha}
	\mathcal{W}_\alpha = (p^{\alpha + 1}_{q+1})^{-1}(\mathcal{W}) \; \cap \; \mathcal{R}^{\alpha}_{\textrm{tang}} \subset J^{\alpha+1}(\mathbb{R}, \mathbb{R}^N).
\end{equation}
Each $\mathcal{W}_\alpha$ has the same $C^\infty$-solutions for $\alpha \ge 2q$, namely, the $\mathcal{W}$-regular $\mathcal{D}$-horizontal curves. Denote the sheaves 
\begin{equation}
	\Phi^\mathcal{W} = \Sol \mathcal{W}_\alpha \quad \text{and} \quad \Psi_\alpha^\mathcal{W} = \Gamma \mathcal{W}_\alpha, \qquad \text{for $\alpha \ge 2q$.}
\end{equation}
A direct application of the results in \cite[pg. 118-120]{gromovBook} then gives us the following.
\begin{theorem}\label{thm:microflexibleLocalWHE}
	Fix $q$ satisfying $nq \ge p - n$, where $p = \cork \mathcal{D}, n = \rk \mathcal{D}$. Then the following holds.
	\begin{itemize}
		\item $\Phi^\mathcal{W}$ is microflexible.
		\item For any \[\alpha \ge \max\{2q + 1 + q, \; 2.1 + 2q \} = 3q + 1,\] the jet map $j^{\alpha + 1} : \Phi^\mathcal{W} \to \Psi^\mathcal{W}_\alpha$ is a local weak homotopy equivalence. 
	\end{itemize}
\end{theorem}

\begin{remark} \label{rmk:calWIndependentOfChoice}
	It should be noted that the $\mathcal{W}$-regularity of $\mathcal{D}$-horizontal maps $\mathbb{R} \to M$ is independent of any choice of coordinates on $M$ or choice of defining $1$-forms $\lambda^s$. Indeed, these are precisely the class of maps $\mathbb{R} \to M$ over which the operator $\mathfrak{D}$ (\autoref{eqn:opD}) is infinitesimally invertible. Since the solution space $\mathfrak{D}=0$ is independent of any choice, so is the regularity of such maps.
\end{remark}

\subsection{Local $h$-Principle for $\mathcal{W}$-Regular Horizontal Curves} To keep the notation light, throughout the rest of this section, we shall treat any higher jet $j^q_u(x)$ formally as variables. That is, $j^q_u(x)$ really represents the tuple of \emph{formal} maps $\big(F^i : \odot^i T_x\mathbb{R} \to T_{u(x)}\mathbb{R}^N, 1\le i\le q\big)$ in the jet space $J^q_{(x,u(x))}(\mathbb{R},\mathbb{R}^N)$, and each component $\partial_t^i u(x) \equiv F^i(\partial_t^i) \in \mathbb{R}^N$ are independent variables. For any $1$-form $\lambda$ defining $\mathcal{D}$ near $y = u(x)$, the components of the higher jets $j^q_{\lambda}(y)$ will be treated as known scalar values.

Now, consider the first order relation
\begin{equation}
	\mathcal{R}_{\textrm{imm-tang}} = \Big\{j^1_u(x) \in J^1(\mathbb{R}, \mathbb{R}^N) \;\Big|\; \text{$du_x$ is injective and $\im du_x \subset \mathcal{D}_{u(x)}$}\Big\}.
\end{equation}
The solution sheaf $\Sol \mathcal{R}_{\textrm{imm-tang}}$ consists of all the $\mathcal{D}$-horizontal immersed curves. It follows from \autoref{eqn:rel:tangentAlpha} that $\mathcal{R}_{\textrm{imm-tang}} \subset \mathcal{R}^0_{\textrm{tang}}$. For any $\alpha \ge q$, where $q$ satisfies $nq \ge p - n$, we have from \autoref{eqn:rel:calWAlpha} that the jet projection map $p^{\alpha+1}_1 : J^{\alpha+1}(\mathbb{R},\mathbb{R}^N) \to J^1(\mathbb{R},\mathbb{R}^N)$ restricts to a map \[p^{\alpha + 1}_1|_{\mathcal{W}_\alpha} : \mathcal{W}_\alpha \to \mathcal{R}_{\textrm{imm-tang}}.\]
In fact, we have the following commutative diagram.
\[\begin{tikzcd}
	\mathcal{W}_\alpha \arrow{d}[swap]{p^{\alpha+1}_1|_{\mathcal{W}_\alpha}} \arrow[hookrightarrow]{r} & \mathcal{R}^\alpha_{\textrm{tang}} \arrow{d}{p^{\alpha+1}_1|_{\mathcal{R}^\alpha_{\textrm{tang}}}} \arrow[hookrightarrow]{r} & J^{\alpha+1}(\mathbb{R}, \mathbb{R}^N) \arrow{d}{p^{\alpha+1}_1}\\
	\mathcal{R}_{\textrm{imm-tang}} \arrow[hookrightarrow]{r} & \mathcal{R}^0_{\textrm{tang}} \arrow[hookrightarrow]{r} & J^1(\mathbb{R},\mathbb{R}^N)
\end{tikzcd}\]
Note that $p^{\alpha+1}_1$ is an affine bundle. Let us now make the following easy observation.
\begin{lemma}\label{lemma:infinitesimalLift}
	For any $\sigma \in \mathcal{R}^0_{\textnormal{tang}}$, the fiber of $p^{\alpha+1}_1|_{\mathcal{R}^\alpha_{\textnormal{tang}}}$ over $\sigma$ is contractible.
\end{lemma}
\begin{proof}
	A jet $j^{\alpha+1}_u(x) \in \mathcal{R}^\alpha_{\textrm{tang}} \subset J^{\alpha+1}(\mathbb{R},\mathbb{R}^N)$ is defined by the equation $j^\alpha_{\mathfrak{D}(u)}(x)=0$ (\autoref{eqn:rel:tangentAlpha}), which expands to the following system:
	\begin{equation}\label{eqn:infSol:mainEqn}
		\left.\begin{aligned}
			&0 = \partial_t^k\big(u^*\lambda^s(\partial_t)\big)\big|_x = \partial_t^k \big((\lambda^s_\mu \circ u) \partial_t u^\mu\big)\big|_x \\
			\Rightarrow &\big(\lambda^s_\mu \circ u\big) \partial_t^{k+1} u^\mu\big|_x + \text{terms involving $j^k_u(x)$} = 0, \quad 1 \le s \le p, \; 0 \le k \le \alpha.		
		\end{aligned}\right\}
	\end{equation}
	Recall the matrix $\Lambda = \begin{pmatrix} \lambda^s_\mu \circ u (x) \end{pmatrix}_{p \times N}$ from \autoref{eqn:R_uLambdaNotation} and denote by $\partial_t^k u$ the column vector $\partial_t^k u = \begin{pmatrix} \partial_t^k u^\mu(x) \end{pmatrix}_{N\times 1}$. Then, \autoref{eqn:infSol:mainEqn} can be expressed as the following affine system:
	\begin{equation}\label{eqn:infSol:mainSystem}
		\left.
		\begin{aligned}
			\Lambda \partial_t u &= 0 \\
			\Lambda \partial_t^2 u &= \begin{pmatrix} - (\partial_\nu \lambda^s_\mu \circ u) \partial_t u^\nu \partial_t u^\mu \end{pmatrix}_{p\times 1}\Big|_x \\
			&\vdots\\
			\Lambda \partial_t^{\alpha+1} u &= \text{$p\times 1$ vector involving $j^\alpha_u(x)$}
		\end{aligned}
		\right\}
	\end{equation}
	Note that $\mathcal{R}^0_{\textrm{tang}}|_{(x,u(x))} = \ker\Lambda$. Since $\Lambda$ has full rank, given any value of $\sigma = j^1_u(x) \equiv \partial_t u \in \mathcal{R}^0_{\textrm{tang}}$, the above system can always be solved in a triangular way. Clearly, at each step the solution space is affine. It follows that the fiber of $p^{\alpha+1}_1|_{\mathcal{R}^\alpha_{\textrm{tang}}}$ over $\sigma$ is contractible.
\end{proof}

The discussion so far had no extra assumption on the distribution $\mathcal{D}$. From this point onwards, we shall only consider equiregular, bracket-generating distributions. The main goal of this section is to prove the following.

\begin{theorem}\label{thm:jetLifting}
	Let $\mathcal{D}$ be an equiregular bracket-generating distribution of rank $n$ and corank $p$ on $\mathbb{R}^{N = n+p}$, with type $\mathfrak{m}$. Let $q_0$ satisfy $n q_0 \ge p-n$. Fix a jet $\sigma = j^1_u(x) \in \mathcal{R}_{\textnormal{imm-tang}}$. Then for each $K \ge 1$, there exists some $q(\mathfrak{m}, K) \ge q_0$, such that for any $\alpha \ge q(\mathfrak{m}, K)$ the complement of the fiber $\mathcal{W}_\alpha|_{\sigma} = \left( p^{\alpha+1}_1|_{\mathcal{W}_\alpha} \right)^{-1}(\sigma)$ has codimension at least $K$ in $\mathcal{R}^\alpha_{\textnormal{tang}}|_{\sigma} = \left( p^{\alpha+1}_1|_{\mathcal{R}^\alpha_{\textnormal{tang}}} \right)^{-1}(\sigma)$.
\end{theorem}

As a corollary, we get the local $h$-principle for $\mathcal{W}$-regular horizontal maps.
\begin{corr}\label{corr:localWHE}
	The sheaf map $\Phi^\mathcal{W} \to \Gamma \mathcal{R}_{\textnormal{imm-tang}}$ induced by the differential map is a local weak homotopy equivalence.
\end{corr}
\begin{proof}
	Fix some jet $\sigma \in \mathcal{R}_{\textrm{imm-tang}}$. It follows from \autoref{thm:jetLifting} and \autoref{lemma:infinitesimalLift} that the fiber of $p^{\alpha+1}_1|_{\mathcal{W}_\alpha}$ over $\sigma$ is $K$-connected for $\alpha$ sufficiently large. Hence, passing to the infinity jet, we get the fiber is weakly contractible. By an argument of Gromov \cite[pg. 77]{gromovBook}, the sheaf map $p^\infty_1 : \Psi^\mathcal{W}_\infty = \Gamma \mathcal{W}_\infty \to \Gamma \mathcal{R}_{\textrm{imm-tang}}$ is then a local weak homotopy equivalence. Also, from \autoref{thm:microflexibleLocalWHE} we have the sheaf map $j^\infty : \Phi^\mathcal{W}  \to \Psi^\mathcal{W}_\infty$ is a local weak homotopy equivalence. But then the composition $p^\infty_1 \circ j^\infty : \Phi^\mathcal{W} \to \Gamma \mathcal{R}_{\textrm{imm-tang}}$ is a local weak homotopy equivalence as well. Note that the composition map is nothing but the differential map $u \mapsto du$.
\end{proof}

To prove \autoref{thm:jetLifting}, we need to understand the equations involved in defining the relation $\mathcal{W}_\alpha \subset J^{\alpha + 1}(\mathbb{R}, \mathbb{R}^N)$ as in \autoref{eqn:rel:calWAlpha}. We have already seen in \autoref{lemma:infinitesimalLift} that given a jet $\sigma \in \mathcal{R}_{\textrm{imm-tang}}$, the fiber of $\mathcal{R}^{\alpha}_{\textrm{tang}}|_\sigma$ is the solution space of a triangular affine system (\autoref{eqn:infSol:mainSystem}). But a jet $j^{\alpha+1}_u(x) \in \mathcal{W}_\alpha|_\sigma \subset \mathcal{R}^\alpha_{\textrm{tang}}|_\sigma$ must satisfy $\mathcal{W}$-regularity as well, i.e, the matrix $A = A(j^{\alpha+1}_u(x))$ as given in \autoref{matrix:original} must have independent rows. We note the following features of the matrix that will become useful later in the proof.
\begin{itemize}
	\item The $\Lambda$ blocks in the off-diagonal of $A$ have full rank. Thus, the rank can only drop at the first row-block. In any block above the $\Lambda$-diagonal, the highest order jet term $\partial_t^{q+1}u$ is contributed by the $\partial_t^q R_u$ factor, and it appears linearly. In fact, from \autoref{eqn:R_uLambdaNotation} we have 
		\begin{equation}\label{eqn:partial_t_q_Ru_original}
			\partial_t^q R_u = \begin{pmatrix} d\lambda^s(\partial_t^{q+1} u, \partial_\mu)\end{pmatrix}_{p\times N} + \text{a $p\times N$ matrix in $j^q_u(x)$}. 
		\end{equation}
		Furthermore, no component of $\partial_t^{q+1} u$ appears anywhere below the diagonal passing through this block (see \autoref{fig:highestOrderJet}). In particular, in each column-block, the highest order derivative of $u$ occurs in the first row-block only.
		{\tiny
		\begin{figure}[h]
			\centering
    \def\svgwidth{0.5\columnwidth}
    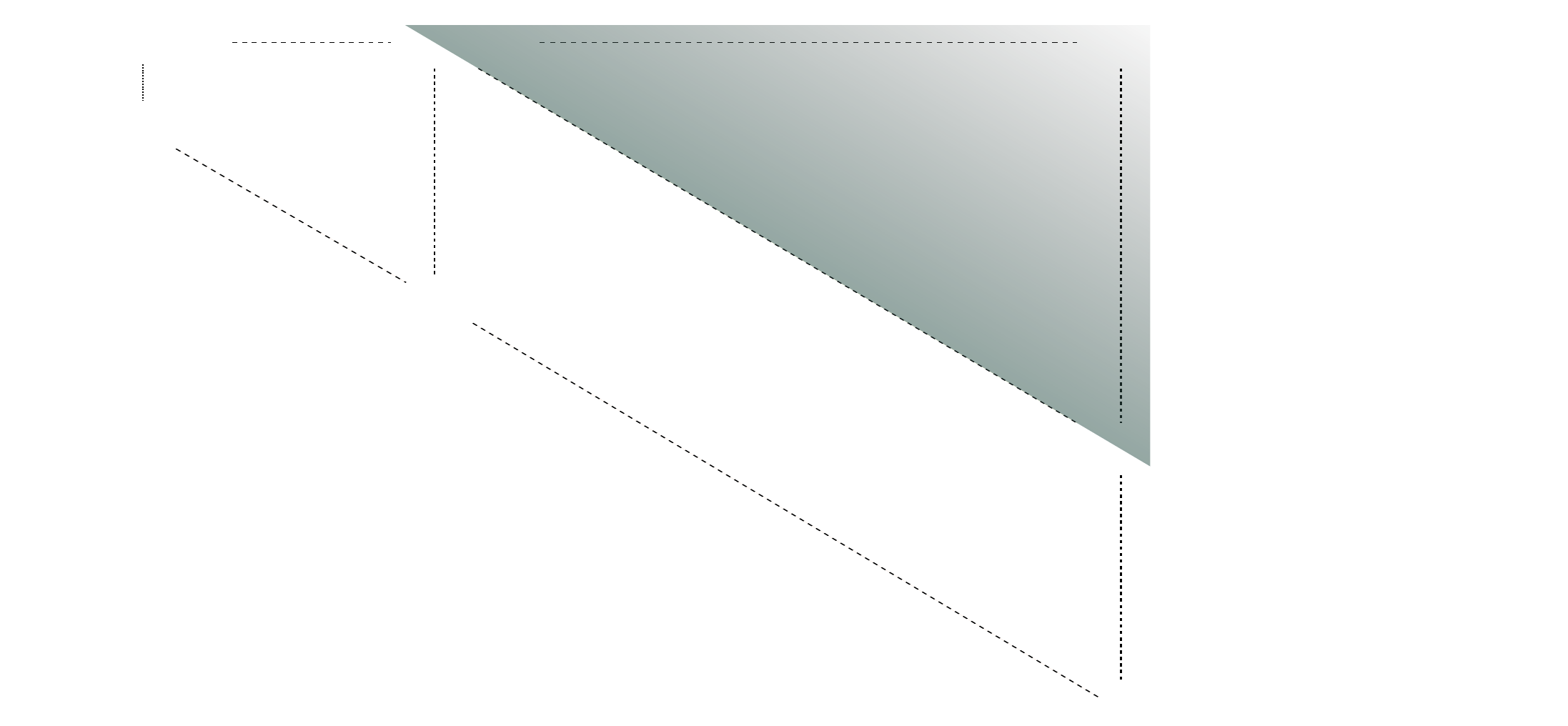

			\caption{Highest order jet of $u$ in $A$}
			\label{fig:highestOrderJet}
		\end{figure}}
	
	\item Each column-block of $A$ has $N$ many columns, which can be labeled by the framing $\{\partial_1,\ldots,\partial_N\}$ of $T\mathbb{R}^N$, as it is clear from \autoref{eqn:partial_t_q_Ru_original}. For any arbitrary choice of frame $\{W_1, \ldots, W_N\}$, we can always perform some (invertible) column operations on $A$ so that the columns in the target column-block, say the $(q+1)^{th}$-column-block, are now labeled by $\{W_1, \ldots, W_N\}$. Indeed, if we write $W_i = W_i^j \partial_j$, then one can consider the invertible matrix $W = \begin{pmatrix} W_1^j & \dots & W_N^j\end{pmatrix}_{N\times N}$, so that multiplying the $\partial_t^q R_u$ block from the right by $W$ transforms it into the following block
	\begin{equation}\label{eqn:partial_t_q_Ru}
		\begin{pmatrix} d\lambda^s(\partial_t^{q+1} u, W_\mu)\end{pmatrix}_{p\times N} + \text{a $p\times N$ matrix in $j^q_u(x)$}.
	\end{equation}
	We extend this $W$ to a block-diagonal-matrix $\tilde{W}$ of size $N(\alpha+1) \times N(\alpha+1)$ by putting $W$ as the $(q+1)^{th}$-diagonal block, and $\mathrm{Id}_N$ in all the other diagonal positions. Now, if we multiply the matrix $A$ by $\tilde{W}$ from the right, it will perform column manipulations precisely at the $(q+1)^{th}$-column-block. In particular, the top row-block in this column is now $\left( \partial_t^q R_u \right) W$, and thus by looking at \autoref{eqn:partial_t_q_Ru}, we may label this column-block by $W_1, \dots, W_N$. Note that this process does not change the rank of the matrix, since the column manipulation is invertible by construction.
\end{itemize}

In the proof, for each column-block of $A$, we shall only prescribe a sub-frame of $T\mathbb{R}^N$ (obtained from using \autoref{lemma:specialVectorFields}), which will then be arbitrarily extended to a full frame. Performing the column manipulation as described above will make sure that the target column-block is labeled first by the prescribed sub-frame and then by the arbitrary choice of extension. As we shall see, we are not interested in the columns which are labeled arbitrarily during this process. If the matrix has full (row) rank after discarding a few columns, then the original matrix will also have full rank. Thus, given a sub-frame, say, $(W_1,\dots, W_t)$ for the $(q+1)^{th}$-column-block, we shall say that {\it the $(q+1)^{\textrm{th}}$-column-block is relabeled by the sub-frame}, while discarding the arbitrarily extended part.

Let us now proceed with the proof of \autoref{thm:jetLifting}. We refer to \autoref{sec:toyCase}, where the major steps of the proof are carried out in detail for an example case.

\begin{proof}[Proof of \autoref{thm:jetLifting}]
	Let $\sigma = j^1_u(x) \in \mathcal{R}_{\textrm{imm-tang}}$ be a given jet, and $y = u(x)$. Suppose $\mathcal{D}$ has the type $\mathfrak{m} = (0=m_0 < \cdots < m_{r+1} = N=n+p)$, where $m_s = \dim \mathcal{D}^s_{y}$ for $0 \le s \le r+1$. We denote $p_s = \dim(\mathcal{D}^{s+1}_y/ \mathcal{D}^s_y) = m_{s+1} - m_s$ for $1 \le s \le r$ and set $p_0 = 0$. Using \autoref{lemma:specialVectorFields}, we get the vectors \[\tau^{s,j} \in \mathcal{D}_y, \; \eta^{s,j} \in \mathcal{D}_y^{s} \setminus \mathcal{D}_y^{s-1}, \; \zeta^{s,j} \in \mathcal{D}^{s+1}_y \setminus \mathcal{D}^s_y, \qquad 1 \le j \le p_s,\quad 1 \le s \le r\] at $y$, and the $1$-forms $\lambda^{s,j}$ near $y$. We write the matrix $A$ (\autoref{matrix:original}) in terms of these $\lambda^{s,j}$'s. As observed in \autoref{rmk:calWIndependentOfChoice}, this does not change the relations $\mathcal{W}_\alpha$.\\
	
	\noindent\underline{\it Notations:} We label the row and column-blocks of $A$ starting from $0$, so that the $(0,q)^{\text{th}}$ block is $\partial_t^q R_u$ and the $(q+1, q)^\text{th}$-block is the $\Lambda$ block. We will use the notation $\zeta^{s,\bullet}$ to mean the tuple of vectors $(\zeta^{s,1},\ldots, \zeta^{s,p_s})$ and similarly $\tau^{s,\bullet}, \eta^{s,\bullet}, \lambda^{s,\bullet}$ etc. In particular, the matrix $\Lambda$ is then given by $\Lambda^\dagger = \begin{pmatrix} \lambda^{1,\bullet} & \cdots & \lambda^{p,\bullet} \end{pmatrix}_{N \times p}$. We shall also use the notation $\zeta^{\bullet}$ for the tuple $\big(\zeta^{1,\bullet},\ldots,\zeta^{r,\bullet}\big)$ of size $p$, and $\hat\zeta^{s,\bullet}$ for $1 \le s \le r$ for the tuple of size $p-p_s$, obtained by dropping the tuple $\zeta^{s,\bullet}$ from $\zeta^{\bullet}$. For notational convenience, we set $\hat\zeta^{0,\bullet} = \zeta^\bullet$.\\
	
	Let us first assume $K=1$. We need to show that the complement of $\mathcal{W}_\alpha|_{\sigma}$ in $\mathcal{R}^\alpha_{\textrm{tang}}|_\sigma$ has codimension $1$ for some $\alpha$ large enough. To achieve this, we shall find a polynomial $P$ in the jet variables $j^\alpha_u(x)$ so that $P$ being non-zero at some $\tilde{\sigma} \in \mathcal{R}^\alpha_{\textrm{tang}}|_\sigma$ implies that the matrix $A=A(\tilde{\sigma})$ has full (row) rank. Let us briefly discuss the proof strategy.\\
	
	\noindent\underline{Step $1$} \phantomsection\label{thm:jetLifting:step:1} For each column-block of $A$, we shall prescribe some sub-frame of $T_y\mathbb{R}^N$, consisting of some suitable vectors as obtained by \autoref{lemma:specialVectorFields} (and thus using the fact that $\mathcal{D}$ is bracket-generating). This step (\hyperref[algo:subFrame]{Algorithm 1}) is recursive, and it will determine the value of $q(\mathfrak{m}, 1)$. Extending each sub-frame arbitrarily to a full frame and then relabeling the column-blocks, we get a new matrix, say, $A_1$. Since these operations are invertible, $A$ has full rank if and only if $A_1$ has full rank. \\

	\noindent\underline{Step $2$} \phantomsection\label{thm:jetLifting:step:2} We shall consider the sub-matrix $B$ of $A_1$, with columns labeled by the prescribed sub-frames as above and all the rows of $A_1$. $B$ will be a square matrix and $P = \det B$ will be our candidate polynomial in the higher jet variables $j^{\alpha+1}_u(x)$. Since it is a minor of $A_1$, $\det B \ne 0$ at some higher jet $\tilde{\sigma} \in \mathcal{R}^\alpha_{\textrm{tang}}|_\sigma$ implies that $A(\tilde{\sigma})$ has full rank, i.e, $\tilde\sigma \in \mathcal{W}_\alpha|_\sigma$. We shall keep using the notation $\Lambda$ and $\partial_t^q R_u$ to denote the respective blocks in $B$ obtained after the column transformations and curtailing of $A$. Performing some more (invertible) row and column operations on $B$, we shall produce a new matrix $B_1$, so that $\det B = \det B_1$. Next, we shall extract a square sub-matrix $C$ of $B_1$ and observe that $\det B = \det B_1 = \pm \det C$.\\
	
	\noindent\underline{Step $3$} \phantomsection\label{thm:jetLifting:step:3} It follows from \autoref{eqn:infSol:mainSystem} that if $j^q_u(x)$ is solved, then the solution space for $\partial_t^{q+1} u$ is given as the \emph{affine} space $V_q+ \ker \Lambda = V_q + \mathcal{D}_{u(x)}$, where the $N \times 1$ vector $V_q = V_q(j^q_u(x))$ is obtained using some fixed choice of right inverse $\Lambda^{-1}$. In particular, we can write $\partial_t^{q+1} u = X_q \tau^q + V_q(j^q_u(x))$ for $1 \le q \le \alpha$, where $X_q$ is some indeterminate and $\tau^q$ is a vector suitably set in \hyperref[algo:subFrame]{Algorithm 1} to either $0 \in \mathcal{D}_{y}$ or to one of the vectors $\tau^{s,j} \in \mathcal{D}_{y}$ chosen earlier. Inductively, we then have
	\begin{equation}\label{eqn:partial_t_q_u}
		\partial_t^{q+1} u = X_q \tau^q + \text{terms in $X_1,\ldots, X_{q-1}, \; \tau^1,\ldots,\tau^{q-1}$ and $\sigma = j^1_u(x)$, \quad $1\le q\le \alpha$.}
	\end{equation}
	Arbitrary values of $X_1,\ldots, X_\alpha$ will produce $j^{\alpha+1}_u(x) \in \mathcal{R}^\alpha_{\textrm{tang}}|_\sigma$ from \autoref{eqn:partial_t_q_u}. We will replace the values of $\partial_t^{q+1} u$ in the matrix $C$. Treating $\det C$ as a polynomial in the indeterminates $X_q$ we shall show that $\det C$ is non-vanishing for suitably large values of $X_q$. Thus, $P = \det B$ is non-vanishing when restricted to $\mathcal{R}^\alpha_{\textrm{tang}}|_\sigma$. This will conclude the proof for $K=1$.\\
	
	The crux of the proof lies in suitably choosing the sub-frames for each of the column-blocks of $A$ as done in \hyperref[thm:jetLifting:step:1]{Step $1$}. We produce a schematic diagram to explain the process (\autoref{fig:algorithm}). As discussed earlier, the rank can only drop in the first row-block consisting of $\partial_t^q R_u$, which are represented by the red boxes in \autoref{fig:algorithm}, whereas the blue boxes represent the $\Lambda$ blocks. Suppose we have dealt with the rows corresponding to $d\lambda^{1,1}, \dots d\lambda^{s^\prime,j^\prime}$ of the first row-block in a manner that, ignoring the rest of the rows from the first row-blcok, the matrix $A$ up to, say, the $q^{th}$-column-block has full rank. We could keep choosing $\zeta^\bullet$ for the subsequent column-blocks, which will transform all the $\Lambda$ blocks appearing in those column-blocks into $\mathrm{Id}_p$, and thus make sure that the matrix $A$, after \emph{ignoring} the rest of the rows from the first row-block, is indeed full rank.
	
	{\tiny
		\begin{figure}[h]
			\centering
    \def\svgwidth{0.5\columnwidth}
    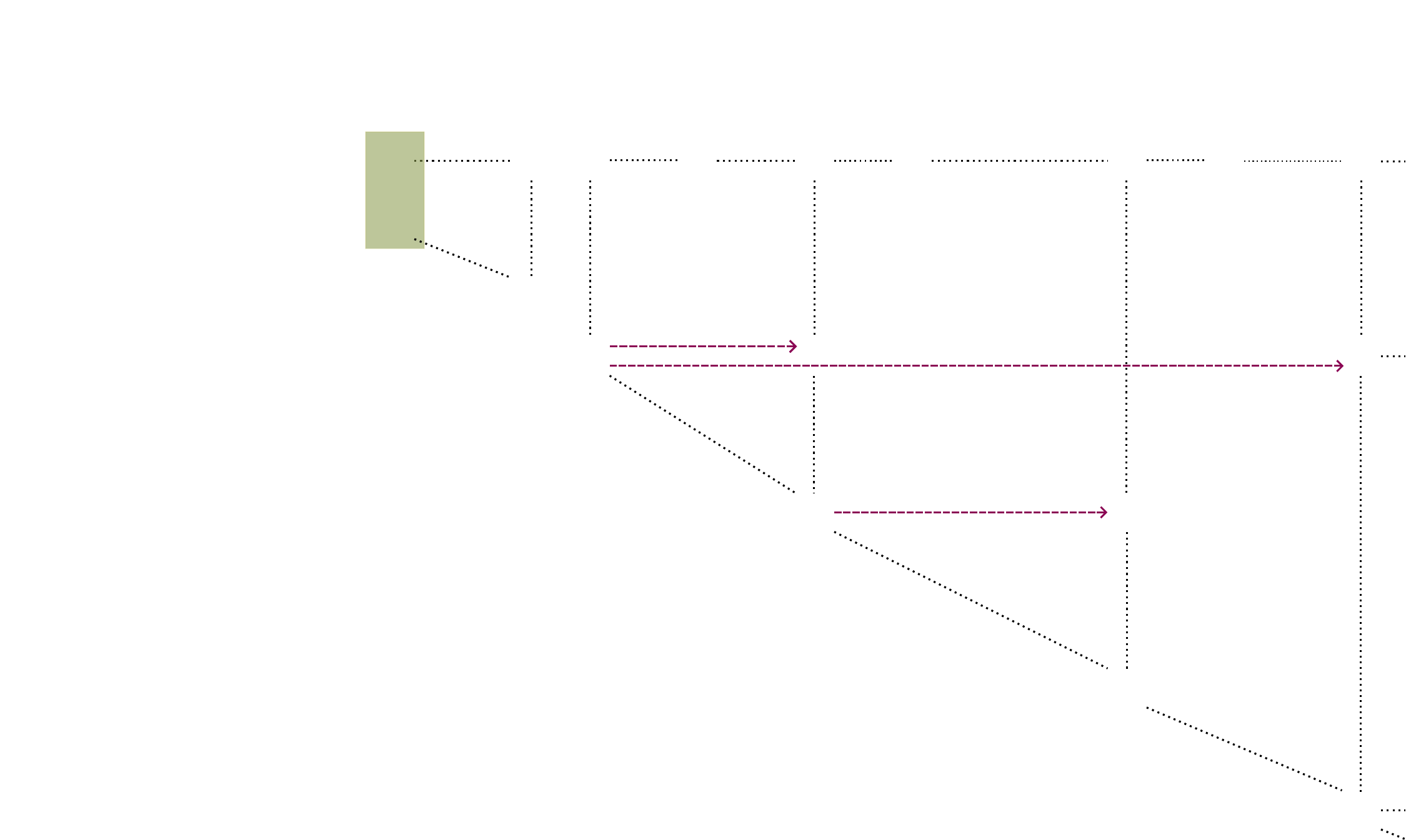

			\caption{How \hyperref[algo:subFrame]{Algorithm 1} chooses the labeling sub-frames.}
			\label{fig:algorithm}
	\end{figure}}
	Now, suppose the next row in the first row-block corresponds to $d\lambda^{s, j}$. We label the $(q+1)^{th}$-column-block by the frame $(\hat{\zeta}^{s-1,\bullet},\eta^{s,j})$. It follows from \autoref{eqn:partial_t_q_Ru} that the last column of the $\partial_t^q R_u$ block transforms into 
	\[d\lambda^{s,j}(\partial_t^{q+1}u, \eta^{s,j}) + \text{a $p \times 1$ vector in $j^q_u(x)$}.\]
	Since $\partial_t^{q+1} u$ is a jet that did not appear earlier in the matrix (\autoref{fig:highestOrderJet}), we can prescribe its value arbitrarily to make sure that (at least) the row corresponding to $d\lambda^{s,j}$ in $\partial_t^q R_u$ must be linearly independent. Indeed, it follows from \autoref{lemma:infinitesimalLift}, that given $j^q_u(x)$, we may choose $\partial_t^{q+1} u$ from an \emph{affine} space, which is parallel to the distribution $\mathcal{D} = \ker \Lambda$, and so, given a particular solution of $\partial_t^{q+1} u$, we can always add a vector proportional to $\tau^q = \tau^{s,j} \in \mathcal{D}^s$. This is done in \hyperref[thm:jetLifting:step:3]{Step $3$}, by adding $X_q \tau^q$ for large value of $X_q$. It follows from \autoref{lemma:infinitesimalLift} that this will introduce an $X_q$ variable (with coefficient $1$) at the $d\lambda^{s,j}$ row, which does not appear anywhere below this row. Note that the rows appearing above might have some instances of $X_q$, but these rows will be taken care of by some $X_{q^\prime}$ variable appearing earlier. In other words, each row gets assigned a unique $X_q$.

	But choosing this sub-frame $\left( \hat{\zeta}^{s-1,j}, \eta^{s,j} \right)$ also introduces a column vector of unknown scalars in the corresponding $\Lambda$-block at the $\lambda^{s - 1,\bullet}$-rows (\autoref{eqn:badLambdaBlock}), and in turn (possibly) reduces the rank for the corresponding row-block. To compensate for this, we choose a sufficiently high jet of $u$ that we have not used (recall that we have utilized the jet $\partial_t^{q+1} u$ so far). Indeed, we may consider $\partial_t^{q^\prime} u$ for some $q^\prime >> q$, and look at the block to the right of this $\Lambda$ block where this jet appears for the first time in this row-block. This is represented by a red dashed arrow in \autoref{fig:algorithm}, the arrowhead pointing to the block (the green box in the figure) which is used to compensate for the drop of rank in the $\Lambda$ block. It follows from \autoref{matrix:original} that this block must look like $c\partial_t^{q^\prime - 1} R_u + \text{a $p\times N$ matrix in $j^{q^\prime-1}_u(x)$}$, for some positive integer $c$. We now perform the same process as above: we choose another sub-frame, say,  $\left( \hat{\zeta}^{s - 2, j^\prime}, \eta^{s - 1, j^\prime} \right)$ for the corresponding column-block, choose a value of $\partial_t^{q^\prime - 1} u$, suitably add some vector proportional to $\tau^{s - 1,j^\prime}$, and thus make sure that at least one of the rows of this row-block is now independent. But now we need to keep doing this recursively, as this will again drop the rank of some $\Lambda$ block down the line. The recursion ends when we choose the frame $(\zeta^\bullet, \eta^{1,j})$, since this frame will transform the corresponding $\Lambda$-block into $\begin{pmatrix} \mathrm{Id}_p & 0_{p \times 1} \end{pmatrix}$, which already has full rank.

	We give a recursive algorithm (\hyperref[algo:subFrame]{Algorithm 1}) to choose the appropriate sub-frames for some $1 \le s \le r$ and some $1 \le j \le p_s$ starting from some column-block $q$, while at the same time suitably fixing the vectors $\tau^{q^\prime} \in \mathcal{D}_y$ for $q^\prime \ge q$ that needs to be used later in \hyperref[thm:jetLifting:step:3]{Step $3$}. The inputs of the algorithm correspond to the situation described above: we have dealt with all the rows appearing before the $d\lambda^{s, j}$-row in the first row-block, using up till the $(q-1)^{th}$-column-block. The integer $d$ corresponds to the depth of the recursion; we begin at $d = 0$, and then $d$ increases as we compensate for subsequent the $\Lambda$ blocks, as discussed above.
	
	\begin{algorithm}[h]\label{algo:subFrame}
		\KwIn{$q, s, j, d$ } 
		\KwOut{$q$ } 
		
		\SetKwFunction{FMain}{ChooseSubFrame}
		\SetKwProg{Fn}{Function}{:}{}
		\Fn{\FMain{$q, s, j, d$}}
		{
			\eIf{$d = 0$}
			{
				\tcp{We are in row-block $0$} 
				$\tau^q \longleftarrow \tau^{s,j}$ \\
			}
			{
				\tcp{We are in row-block $d$}
				$\tau^q \longleftarrow 0$ \\
				$\tau^{q - d} \longleftarrow \tau^{s,j}$
			}
			Pick sub-frame $(\hat\zeta^{s-1,\bullet}, \eta^{s,j})$ for column-block $q$ \\
			$q_0 \longleftarrow q$ \\
			\tcp{If $s = 1$, then $p_{s-1} = 0$ and we do not enter the following loop}
			\For{$1 \le a \le p_{s-1}$}
			{
				\For{$1 \le b \le q_0 + 1$}
				{
					Pick sub-frame $\zeta^\bullet$ for column-block $q + b$ \\
					$\tau^q \longleftarrow 0 \in \mathcal{D}_y$
				}
				$q \longleftarrow$ \FMain{$q + q_0 + 2, s - 1, a, q_0 + 1$}
			}
			\textbf{Return: } $q$
		}
		\textbf{End Function}\\[.5em]
		\caption{Algorithm for choosing sub-frames and $\tau^q$'s.}
	\end{algorithm}
	
	\hyperref[algo:subFrame]{Algorithm 1} outputs the last column-block $q$ for which the sub-frame and the vector $\tau^q$ have been chosen. Note that for $s \ge 2$, the algorithm is recursive. Whereas for $s = 1$, we choose the frame $(\zeta^{0,\bullet}, \eta^{1,j}) = (\zeta^\bullet, \eta^{1,j})$ and do not enter the for-loop since $p_{s-1} = 0$ for $s = 1$. We can now find all the sub-frames, starting from some column-block $q$ by the \hyperref[algo:allSubFrame]{Algorithm 2}.
	
	\begin{algorithm}[h] \label{algo:allSubFrame}
		\KwIn{$q$}
		\KwOut{$q$}
		
		\SetKwFunction{FMain}{ChooseAllSubFrames}
		\SetKwFunction{FSub}{ChooseSubFrame}
		\SetKwProg{Fn}{Function}{:}{}
		\Fn{\FMain{$q$}}
		{
			\For{$1 \le s \le r$}
			{
				\For{$1 \le j \le p_s$}
				{
					$q \longleftarrow$ \FSub{$q, s, j, 0$} + 1
				}
			}
			\textbf{Return: } $q - 1$
		}
		\textbf{End Function}\\[.5em]
		\caption{Algorithm for choosing all sub-frames.}
	\end{algorithm}
	
	As before, \hyperref[algo:allSubFrame]{Algorithm 2} outputs the last column-block $q$ for which the sub-frame and the vector $\tau^q$ have been chosen. We are now in a position to get a suitable value of $q(\mathfrak{m}, K)$ for $K=1$. First, we choose the frame $\zeta^\bullet$ for each of the column-block $0, \ldots, q_0$, where $q_0\ge 0$ satisfies $nq_0 \ge p - n$ and set $\tau^0 = \ldots = \tau^{q_0} = 0$. Then, let $q(\mathfrak{m},1) = \texttt{ChooseAllSubFrames}(q_0 + 1)$. Let us denote the transformed matrix by $A_1 = A_1\big(j^{q(\mathfrak{m},1)+1}_u(x)\big)$. This concludes \hyperref[thm:jetLifting:step:1]{Step $1$}.
	
	Next, choose the sub-matrix $B$ of $A_1$ which are labeled by the prescribed columns chosen as in \hyperref[thm:jetLifting:step:1]{Step $1$}. Firstly, note that whenever we are choosing the sub-frame $\zeta^\bullet$ for column-block $q$, the $\Lambda$-block in that column-block of $B$ becomes $\mathrm{Id}_p$, since $\lambda^{\bullet}$ is dual to $\zeta^{\bullet}$. On the other hand, choosing the frame $(\hat\zeta^{s-1,\bullet}, \eta^{s,j})$ for some $(s,j)$ with $s > 1$ makes the $\Lambda$ block of the form 
	\begin{equation}\label{eqn:badLambdaBlock}
		\begin{pmatrix} \mathrm{Id}_{p_1 + \cdots + p_{s-2}} & 0 & 0 \\ 0 & 0 & *_{p_{s-1} \times 1} \\ 0 & \mathrm{Id}_{p_{s} + \cdots + p_r} & 0 \end{pmatrix}_{p \times (p - p_{s-1} + 1)},
	\end{equation}
	where $*$ represents the column vector $(\lambda^{s,\bullet}(\eta^{s,j}))$. Lastly, choosing $(\zeta^\bullet, \eta^{1,j})$ transforms the $\Lambda$ block into $\begin{pmatrix} \mathrm{Id}_p & 0_{p\times 1} \end{pmatrix}_{p\times (p+1)}$. It follows that $B$ is a square matrix of size $p(q(\mathfrak{m},1)+2) \, \times \, p(q(\mathfrak{m},1)+2)$.
	
	Let us now perform some invertible row and column operations on $B$, keeping its rank fixed. Starting from the bottom right corner of $B$ and then going towards the top-left corner, we consider each $\mathrm{Id}$ block in the off-diagonal $\Lambda$ block. Next using these $\mathrm{Id}$ blocks in order, we make everything zero first along the columns and then along the rows. Denote the new matrix by $B_1$ and observe that $\det B = \det B_1$. In any non-zero block of $B_1$ above the $\Lambda$ diagonal, the highest order derivative of $u$, say $\partial_t^{q+1} u$, is still contributed by the $\partial_t^q R_u$ factor of this block. In fact, any of these blocks look like 
	\[\begin{pmatrix} 0_{p \times (p^\prime -1)} & c \, d\lambda^{\bullet}(\partial_t^{q+1} u, \eta^{s,j}) + \text{terms in $j^q_u(x)$} \end{pmatrix}_{p\times p^\prime} \]
	for some integer $1 \le c \le q(\mathfrak{m}, 1)$ and some $p^\prime \in \{p + 1, p - p_s + 1, 1\le s \le r\}$. Note that this integer $c$ is the integer coefficient of the respective $\partial_t^q R_u$ factor as in \autoref{matrix:original}. Let $C$ be the square sub-matrix of $B_1$ obtained by removing the rows and columns corresponding to the $\mathrm{Id}$ blocks so that we have $\det B = \det B_1 = \pm \det C$, concluding \hyperref[thm:jetLifting:step:2]{Step $2$}.
	
	The columns of $C$ are precisely those columns corresponding to some $\eta^{s,j}$ chosen earlier via \hyperref[algo:allSubFrame]{Algorithm 2}, whereas the rows of $C$ correspond to each row from the first row-block of $B_1$, and all the rows starting with some $\lambda^{s,j^\prime}(\eta^{s,j})$ from the other row-blocks. Let us now show that $\det C \ne 0$. We replace $\partial_t^{q+1}u = X_q \tau^q + \cdots$ in $C$ by using \autoref{eqn:partial_t_q_u}. By i) the construction of $C$, ii) the choice of $\tau^q$'s in \hyperref[algo:subFrame]{Algorithm 1}, and iii) from \autoref{lemma:specialVectorFields}, it follows that in each row of $C$ there exists a unique column so that the element in this position satisfies the following:
	\begin{itemize}
		\item The element looks like, $c X_q + \text{terms in $X_1,\ldots,X_{q-1}$}$, for some $q$ and some integer $1 \le c \le q(\mathfrak{m}, 1)$.
		\item $X_q$ does not appear anywhere in $C$ to the left of this column.
		\item $X_q$ does not appear anywhere in this column below this row (but may appear above this row).
	\end{itemize}
	Note that, not every variable $X_q$ appears in $C$, since we have set many $\tau^q = 0$. In fact, there are precisely $\bar q$ many variables appearing, where $C$ has the size $\bar q \times \bar q$. Hence, for notational convenience, let us rename the appearing $X_q$'s to $\{Y_1,\ldots, Y_{\bar q}\}$ in the increasing order. We then have the following recursive formula by expanding $\det C$.
	\begin{equation}\label{eqn:detCRecursive}
		\left.\begin{aligned}
			\det C &= Y_{\bar q} f_{\bar q-1}(\sigma, Y_1,\ldots,Y_{\bar q-1}) + g_{\bar q-1}(\sigma, Y_1,\ldots,Y_{\bar q-1}) \\
			f_{\bar q-1}(\sigma, Y_1,\ldots,Y_{\bar q-1}) &= Y_{\bar q-1}f_{\bar q-2}(\sigma, Y_1,\ldots,Y_{\bar q-2}) + g_{\bar q-2}(\sigma, Y_1,\ldots,Y_{\bar q-2}) \\
			&\vdots \\
			f_2(\sigma, Y_1, Y_2) &= Y_2 f_1(\sigma, Y_1) + g_1(\sigma, Y_1) \\
			f_1(\sigma, Y_1) &= Y_1 \det \tilde{C} + g_0(\sigma)
		\end{aligned} \right\}
	\end{equation}
	Above, $f_i, g_i$ are polynomial functions, where $g_0$ depends only on the choice of the first jet $\sigma = j^1_u(x)$. The matrix $\tilde{C}$ has the following property: for each row of $\tilde{C}$, there exists a unique column, so that the element in that position is an integer (corresponding to the coefficient of some $X_q$) and everything below this integer in this column is $0$. Indeed, this column corresponds to the unique column of $C$ with the first occurrence of $X_q$ as discussed above. Then, some column permutation puts the matrix $\tilde{C}$ in an upper triangular form, with non-zero integers in the diagonal. In particular, $\det \tilde{C} \ne 0$. Hence, choosing $Y_1,Y_2,\ldots, Y_{\bar q}$ successively and sufficiently large, we can see from \autoref{eqn:detCRecursive} that $\det C \ne 0$. In other words, we have shown that the polynomial $\det B$ is non-vanishing when restricted to $\mathcal{R}^{q(\mathfrak{m},1)}_{\textrm{tang}}|_{\sigma}$, which concludes \hyperref[thm:jetLifting:step:3]{Step $3$}.
	
	To finish the proof for $K = 1$, let $\alpha > q(\mathfrak{m}, 1)$. First perform \texttt{ChooseAllSub\-Frames($q_0+1$)} as above. Next, for $q(\mathfrak{m}, 1) + 1 \le q \le \alpha$, choose the sub-frame $\zeta^\bullet$ for the column-block $q$ and set $\tau^q = 0$. We can then continue with the rest of the argument as above without any change. In particular, for any $\alpha \ge q(\mathfrak{m}, 1)$, we have the codimension of the complement of $\mathcal{W}_\alpha|_\sigma$ is at least $1$ in $\mathcal{R}^\alpha_{\textrm{tang}}|_\sigma$.
	
	Let us now induct over $K$. Suppose for some $K \ge 1$, we have obtained a suitable $q(\mathfrak{m}, K)$ and some polynomials $P_1,\ldots, P_K$ in the jets $j^{q(\mathfrak{m}, K)}_u(x)$ so that 
	\begin{equation*}
		\resizebox{\linewidth}{!}{$(\mathcal{W}_{q(\mathfrak{m},K)})^c \subset \{P_1 = \cdots = P_K = 0\},\quad \text{and} \quad \codim \{P_1 = \cdots = P_K = 0\} \ge K \text{ in } \mathcal{R}^{q(\mathfrak{m}, K)}_{\textrm{tang}}|_{\sigma}.$}
	\end{equation*}
	We set $q(\mathfrak{m}, K+1) = \texttt{ChooseAllSubFrames}\big(q(\mathfrak{m}, K) + 1\big)$. This will produce a new polynomial $P_{K+1}$ involving jets that were not involved in any of the $P_1,\ldots, P_K$. But then $\codim \{P_1 = \cdots = P_{K+1} = 0\} \ge K+1$ in $\mathcal{R}^{q(\mathfrak{m}, K+1)}_{\textrm{tang}}|_\sigma$. Also by construction, $(\mathcal{W}_{q(\mathfrak{m},K+1)})^c \subset \{P_1 = \cdots = P_{K+1} = 0\}$. Proceeding similarly as in the $K=1$ case, we can finish the inductive step.
	
	This concludes the proof of \autoref{thm:jetLifting}.
\end{proof}

\subsection{A Toy Case}\label{sec:toyCase}
Let us consider a toy case to illustrate the major steps in the proof of \autoref{thm:jetLifting}. We consider a bracket-generating distribution $\mathcal{D}$, of rank $n = 10$ and corank $p = 4$, on a manifold $M$ with $\dim M = 14$. Suppose $\mathcal{D}^3 = T M$ and $\rk\big(\mathcal{D}^2/\mathcal{D}\big) = 2 = \rk\big(\mathcal{D}^3/\mathcal{D}^2\big)$. Thus, $\mathcal{D}$ has the type \[\mathfrak{m} = (0 = m_0 < m_1 = 10 < m_2 = 12 < m_3 = 14),\]
and we have $p_1 = 2 = p_2$. For some $x \in M$, using \autoref{lemma:specialVectorFields}, we choose the necessary vectors $\tau^{s,j} \in \mathcal{D}_x, \eta^{s,j} \in \mathcal{D}^s, \zeta^{s,j} \in \mathcal{D}^{s+1}_x$, for $1 \le j \le p_s$ and $1 \le s \le 2$ and the $1$-forms $\lambda^{s,j}$. Clearly, for $q_0 = 0$ we have $n q_0 > p - n = -6$. Let us now determine $q(\mathfrak{m}, 1)$ and see how \hyperref[algo:allSubFrame]{Algorithm 2} produces the sub-frames for different column-blocks of the matrix $A$.
\begin{itemize}
	\ifarxiv \setlength\itemsep{-0.3em} \fi
	\item Pick $\zeta^\bullet$ for column-block $0$.
	\item Pick $(\zeta^\bullet, \eta^{1,1})$ for column-block $1$. 
	\item Pick $(\zeta^\bullet, \eta^{1,2})$ for column-block $2$. 
	\item Pick $(\hat\zeta^{1,\bullet}, \eta^{2,1})$ for column-block $3$. 
	\item Pick $\zeta^\bullet$ for column-blocks $4$ to $7$.
	\item Pick $(\zeta^\bullet, \eta^{1,1})$ for column-block $8$
	\item Pick $\zeta^\bullet$ for column-blocks $9$ to $12$
	\item Pick $(\zeta^\bullet, \eta^{1,2})$ for column-block $13$
	\item Pick $(\hat\zeta^{1,\bullet}, \eta^{2,2})$ for column-block $14$
	\item Pick $\zeta^\bullet$ for column-blocks $15$ to $29$
	\item Pick $(\zeta^\bullet, \eta^{1,1})$ for column-block $30$
	\item Pick $\zeta^\bullet$ for column-blocks $31$ to $45$
	\item Pick $(\zeta^\bullet, \eta^{1,2})$ for column-block $46$
\end{itemize}
Thus, we may take $q(\mathfrak{m}, 1) = 46$. The sub-matrix $B$ is a square matrix of size $4 \times (46+2) = 192$. We also choose the vectors $\tau^q$ as follows:
\begin{align*}
	\tau^1 = \tau^{1,1}, \quad \tau^2 = \tau^{1,2}, \quad \tau^3 = \eta^{2,1}, \quad \tau^4 = \tau^{1,1}, \\
	\quad \tau^9 = \tau^{1,2}, \quad \tau^{14} = \tau^{2,2}, \quad \tau^{15} = \tau^{1,1}, \quad \tau^{31} = \tau^{1,2}
\end{align*}
and set all other $\tau^q = 0$ for $0 \le q \le 46$.

Removing all the rows and columns corresponding to some $\mathrm{Id}$ block in the off-diagonal, we get the matrix $C$ of size $8 \times 8$ from $B$. Let us also replace $\partial_t^{q+1} u = X_q \tau^q + \cdots$ in $C$. Then $C$ looks as follows:

\medskip
\resizebox{.98\linewidth}{!}{$C = \begin{blockarray}{cccccccc}
	\eta^{1,1} & \eta^{1,2} & \eta^{2,1} & \eta^{1,1} & \eta^{1,2} & \eta^{2,1} & \eta^{1,1} & \eta^{1,2} \\
	\begin{block}{(cccccccc)}
		X_1 + * & ?X_2 + * & ?X_3 + * & ?X_4 + * & ?X_9 + * & ?X_{14} + * & ?X_{15} + * & ?X_{31} + * \\
		0X_1 + * & X_2 + * & ?X_3 + * & ?X_4 + * & ?X_9 + * & ?X_{14} + * & ?X_{15} + * & ?X_{31} + * \\
		0X_1 + * & 0X_2 + * & X_3 + * & ?X_4 + * & ?X_9 + * & ?X_{14} + * & ?X_{15} + * & ?X_{31} + * \\
		0X_1 + * & 0X_2 + * & 0X_3 + * & ?X_4 + * & ?X_9 + * & X_{14} + * & ?X_{15} + * & ?X_{31} + * \\
		0 & 0 & \lambda^{2,1}(\eta^{2,1}) & X_4 + * & ?X_9 + * & 0X_{14} + * & ?X_{15} + * & ?X_{31} + * \\
		0 & 0 & \lambda^{2,2}(\eta^{2,1}) & 0X_4 + * & X_9 + * & 0X_{14} + * & ?X_{15} + * & ?X_{31} + * \\
		0 & 0 & 0 & 0 & 0 & \lambda^{2,1}(\eta^{2,1}) & X_{15} + * & ?X_{31} + * \\
		0 & 0 & 0 & 0 & 0 & \lambda^{2,2}(\eta^{2,1}) & 0X_{15} + * & X_{31} + * \\
	\end{block}
\end{blockarray}_{8\times 8}$}
\medskip

\noindent Above, $*$ in some $?X_i + *$ denotes terms involving $X_1,\ldots,X_{i-1}$. Also, to keep the calculation simple, we have set the (non-zero) integer coefficients to $1$ for the highest order $X_q$ uniquely associated with each row. We can find a recursive formula for $\det C$ as in \autoref{eqn:detCRecursive}. In particular, the matrix $\tilde{C}$ is given as
\[
\tilde{C} = \begin{pmatrix}
	1 & ? & ? & ? & ? & ? & ? & ? \\
	0 & 1 & ? & ? & ? & ? & ? & ? \\
	0 & 0 & 1 & ? & ? & ? & ? & ? \\
	0 & 0 & 0 & ? & ? & 1 & ? & ? \\
	0 & 0 & 0 & 1 & ? & 0 & ? & ? \\
	0 & 0 & 0 & 0 & 1 & 0 & ? & ? \\
	0 & 0 & 0 & 0 & 0 & 0 & 1 & ? \\
	0 & 0 & 0 & 0 & 0 & 0 & 0 & 1 \\
\end{pmatrix}_{8 \times 8}
\]
which obviously satisfies $\det \tilde{C} \ne 0$.

\section{The $h$-Principle for Transverse Maps} \label{sec:hPrinTrans}
Given a distribution $\mathcal{D}$, a map $u : \Sigma \to M$ is said to be \emph{transverse} to $\mathcal{D}$ if the composition map $T\Sigma \overset{du}{\longrightarrow} u^*TM \overset{\lambda}{\longrightarrow} u^*TM/\mathcal{D}$ is surjective, where $\lambda : TM \to TM/\mathcal{D}$ is the quotient map. We have the following theorem.

\begin{theorem} \label{thm:hPrinTrans}
	Let $\mathcal{D}$ be an equiregular bracket-generating distribution on a manifold $M$. Then, maps transverse to $\mathcal{D}$ satisfy the $C^0$-dense parametric $h$-principle.
\end{theorem}

Once we have the microflexibility (\autoref{thm:microflexibleLocalWHE}) and the local $h$-principle (\autoref{corr:localWHE}) for $\mathcal{W}$-regular horizontal maps $\mathbb{R} \to M$, the proof essentially follows from the steps outlined as in \cite[pg. 84]{gromovBook}. The details were worked out in \cite[Theorem 14.2.1]{eliashbergBook} when the distribution is contact, and in \cite[Theorem 4]{pinoPresasFlexTangentTransverseEngel} for the case of Engel distributions. We include the sketch of the proof for the general case.

\begin{proof}[Proof of \autoref{thm:hPrinTrans}]
	Let $\mathcal{R}_{\textrm{tran}} = \big\{ j^1_u(x) \; \big| \; \text{$\lambda \circ du_x$ is surjective}\} \subset J^1(\Sigma, M)$ be the relation of $\mathcal{D}$-transverse maps $\Sigma \to M$. Since $\mathcal{R}_{\textrm{tran}}$ is open, we have $\Sol \mathcal{R}_{\textrm{tran}}$ is microflexible and furthermore, $j^1 : \Sol \mathcal{R}_{\textrm{tran}} \to \Gamma \mathcal{R}_{\textrm{tran}}$ is a local weak homotopy equivalence. To prove the $h$-principle, we need to find some suitable (local) microextensions (in the sense of \cite[Definition 5.10]{bhowmickDattaDeg2Fat}) to maps $\tilde{\Sigma} \to M$, where $\tilde{\Sigma} = \Sigma \times \mathbb{R}$.
	
	Let us consider the following class of maps:
	\[ \tilde \Phi^{\textrm{$\mathcal{W}$-tran}} = \Big\{ u: \tilde \Sigma \to M \; \Big| \; \substack{ \text{$u$ is transverse to $\mathcal{D}$, and for each $\sigma \in \Sigma$} \\ \text{$u|_{\sigma \times \mathbb{R}}$ is a $\mathcal{W}$-regular, $\mathcal{D}$-horizontal, immersed curve} } \Big\}.\]
	We prove the microflexibility and local $h$-principle for $\tilde\Phi^{\mathcal{W}-tran}$. Assuming $\mathcal{D} = \cap_{s=1}^p \ker \lambda^s$, let us consider the differential operator 
	\begin{align*}
		\tilde{\mathfrak{D}} : C^\infty(\tilde \Sigma, M) &\to \hom(T\mathbb{R}, \mathbb{R}^p) = C^\infty(\tilde \Sigma, M)\\
		u &\mapsto \big(u^*\lambda^s|_{T\mathbb{R}}\big) = \big(u^*\lambda^s(\partial_t)\big).
	\end{align*}
	Clearly, $u:\tilde{\Sigma} \to M$ is a solution of  $\tilde{\mathfrak{D}} = 0$ precisely when $u|_{\sigma \times \mathbb{R}}$ is $\mathcal{D}$-horizontal for each $\sigma \in \Sigma$. The linearization of $\tilde{\mathfrak{D}}$ at some $u$ is then given by a formula identical to \autoref{eqn:L_uFormula}. Since $t$ is a global coordinate on $\tilde\Sigma = \Sigma \times \mathbb{R}$, we have a splitting of the jet spaces. Denote by $j^{q+1,\perp}_u(x)$ the higher derivatives purely along the $t$ direction. We then see that the matrix $A = A(j^{q+1}_u)$ as in \autoref{matrix:original}, in fact depends only on $j^{q+1,\perp}_u$. Let $\tilde{\mathcal{R}}_{\textrm{tran}} \subset J^1(\tilde{\Sigma}, M)$ be the relation of $\mathcal{D}$-transverse maps $\tilde{\Sigma} \to M$ and define
	\[\mathcal{W}^{\textrm{tran}} = \Big\{j^{q+1}_u(x) \; \Big|\; j^1_u(x) \in \tilde{\mathcal{R}}_{\textrm{tran}}, \; \partial_t u(x) \ne 0, \; \text{$A\big(j^{q+1,\perp}_u(x)\big)$ has full rank}\Big\} \subset J^{q+1}(\tilde \Sigma, M).\]
	By similar arguments as in \autoref{sec:weakRegularity}, the operator $\tilde{\mathfrak{D}}$ is infinitesimally invertible on $\mathcal{W}^{\textrm{tran}}$-regular maps.
	
	Let $\tilde{\mathcal{R}}^\alpha_{\textrm{tang}} = \big\{j^{\alpha + 1}_u(x) \; \big| \; j^\alpha_{\tilde{\mathfrak{D}}(u)}(x) = 0 \big\}$ for $\alpha \ge 0$ and then for $\alpha \ge 2q$ define
	\[\mathcal{W}^{\textrm{tran}}_\alpha = \big(p^{\alpha+1}_{q+1}\big)^{-1}(\mathcal{W}^{\textrm{tran}}) \cap \tilde{\mathcal{R}}^\alpha_{\textrm{tang}} \subset J^{\alpha+1}(\tilde \Sigma, M).\]
	It is then immediate that the smooth solutions of $\mathcal{W}^{\textrm{tran}}_\alpha$ are all same and in fact $\tilde \Phi^{\textrm{$\mathcal{W}$-tran}} = \Sol \mathcal{W}^{\textrm{tran}}_\alpha$. Just as in \autoref{thm:microflexibleLocalWHE}, we have the microflexibility for $\tilde\Phi^{\textrm{$\mathcal{W}$-tran}}$, and $j^{\alpha + 1} : \tilde\Phi^{\textrm{$\mathcal{W}$-tran}} \to \Gamma \mathcal{W}^{\textrm{tran}}_\alpha$ is a local weak homotopy equivalence for $\alpha$ large enough. Since $\tilde\Phi^{\textrm{$\mathcal{W}$-tran}}$ is invariant under the pseudogroup of fiber-preserving local diffeomorphisms of $\tilde \Sigma$, we have the flexibility of the restricted sheaf $\tilde \Phi^{\textrm{$\mathcal{W}$-tran}}|_{\Sigma \times 0}$ \cite[pg. 78]{gromovBook}.
	
	Define the relation 
	\[\tilde{\mathcal{R}}_{\textrm{tang-tran}} = \Big\{j^1_u(x) \in \tilde{\mathcal{R}}_{\textrm{tran}} \; \Big| \; 0 \ne \partial_t u(x) \in \mathcal{D}_{u(x)} \Big\} \subset J^1(\tilde{\Sigma}, M).\]
	The jet projection map $J^{\alpha+1}(\tilde{\Sigma}, M) \to J^1(\tilde{\Sigma}, M)$ restricts to a map $\mathcal{W}^{\textrm{tran}}_\alpha \to \tilde{\mathcal{R}}_{\textrm{tang-tran}}$. Since transversality is a first order regularity condition, we deduce analogous to \autoref{corr:localWHE} that the map $\tilde \Phi^{\textrm{$\mathcal{W}$-tran}} \to \Gamma \tilde{\mathcal{R}}_{\textrm{tang-tran}}$ induced by the differential, is a local weak homotopy equivalence. Now, we have a map $ev : \Gamma \tilde{\mathcal{R}}_{\textrm{tang-tran}} \to \Gamma \mathcal{R}_{\textrm{tran}}$ induced by the restriction $u \mapsto u|_{\Sigma \times 0}$. By choosing some arbitrary non-vanishing local sections of $\mathcal{D}$, we can easily get local (parametric) formal extensions along this $ev$ map on contractible open sets of $\Sigma$. Thus, $\tilde{\mathcal{R}}_{\textrm{tang-tran}}$ is a microextension of $\mathcal{R}_{\textrm{tran}}$. We can now finish the proof of the $h$-principle as in \cite[Theorem 2.18]{bhowmickDattaDeg2Fat}.
\end{proof}

We would like to note that in a recent article \cite{pinoKnotEmbedding}, the authors have proved the above $h$-principle (among many other strong results) under similar constant growth assumption on the distribution, albeit using a completely different rather geometric technique. We believe that the technique utilized in the present article can be adapted to a broader class of problems, including the existence of horizontal immersions of submanifolds. Indeed, this method seems promising to address the following conjecture by Gromov.
\begin{conj*}\label{conj:overregular}\cite[pg. 259]{gromovCCMetric}
	Given a distribution $\mathcal{D}$ on $M$, $\Omega$-regular (i.e., $d\lambda^s$-regular) $\mathcal{D}$-horizontal immersions $\Sigma \rightarrow M$ satisfies the complete $h$-principle, provided $\dim M \ge (\dim \Sigma + 1) \codim \mathcal{D}$.
\end{conj*}

\subsection{Immersions Transverse to a Distribution}
Whenever $\dim \Sigma \le \dim M$, an immersion $u: \Sigma \rightarrow M$ is said to be \emph{transverse to $\mathcal{D}$} if $u$ is an immersion and the composition map $T\Sigma \rightarrow u^* \left( TM/ \mathcal{D} \right)$ is of full rank.
The following $h$-principle is well known.
\begin{theorem}[{\cite[Pg. 87]{gromovBook}, \cite[Pg. 71]{eliashbergBook}}] \label{thm:hPrinTransImmNonCritical}
	Let $\mathcal{D}$ be an arbitrary distribution on $M$. If $\dim \Sigma < \cork \mathcal{D}$, then immersions $\Sigma \to M$ transverse to $\mathcal{D}$ satisfy all forms of $h$-principles.
\end{theorem}
The critical dimension $\dim \Sigma = \cork \mathcal{D}$ is not covered by the above theorem. Although, for the special case of $\mathcal{D}$ being either a contact \cite[Theorem 14.2.2]{eliashbergBook} or an Engel distribution \cite{pinoPresasFlexTangentTransverseEngel}, the $h$-principle holds for all transverse immersions. The $h$-principle for smooth immersions transverse to \emph{real analytic} bracket-generating distributions was proved in \cite{pinoMicroflexibleTransverse}. We have the following.
\begin{theorem}\label{thm:hPrinTransImm}
	Let $\mathcal{D}$ be an equiregular bracket-generating distribution. Then, $C^0$-dense, parametric $h$-principle holds for immersions $\Sigma \to M$ transverse to $\mathcal{D}$, provided $\dim \Sigma < \dim M$.
\end{theorem}
\begin{proof}
	The proof is identical to that of \autoref{thm:hPrinTrans}. Denote by $\mathcal{R}_{\textrm{imm}} \subset J^1(\Sigma, M)$ the relation of immersions $\Sigma \to M$, and similarly $\tilde{\mathcal{R}}_{\textrm{imm}} \subset J^1(\tilde{\Sigma}, M)$ where $\tilde{\Sigma} = \Sigma \times \mathbb{R}$. Then, $\mathcal{R}_{\textrm{imm-tran}} = \mathcal{R}_{\textrm{tran}} \cap \mathcal{R}_{\textrm{imm}}$ is the relation of immersions $\Sigma \to M$ transverse to $\mathcal{D}$. The microextension is provided by the relation $\tilde{\mathcal{R}}_{\textrm{imm-tang-tran}} = \tilde{\mathcal{R}}_{\textrm{tang-tran}} \cap \tilde{\mathcal{R}}_{\textrm{imm}}$, where $\tilde{\mathcal{R}}_{\textrm{tang-tran}}$ is as in \autoref{thm:hPrinTrans}. We have the map $ev : \Gamma \tilde{\mathcal{R}}_{\textrm{imm-tang-tran}} \to \Gamma \mathcal{R}_{\textrm{imm-tran}}$ induced by $u \mapsto u|_{\Sigma \times 0}$. Since $\dim\Sigma < \dim M$, we can always find \emph{non-vanishing} (local) extensions along the $ev$ map. The $h$-principle then follows.
\end{proof}

In particular, we have the $h$-principle for $\mathcal{D}$-transverse maps $\Sigma \to M$ for $\dim \Sigma = \cork \mathcal{D}$, provided $\mathcal{D}$ is bracket-generating. Also, taking $\mathcal{D} = TM$, \autoref{thm:hPrinTransImm} reduces to Hirsch's $h$-principle for immersions $\Sigma \to M$ \cite{hirschImemrsion}.

\begin{remark}
	When $\dim \Sigma \ge \cork \mathcal{D}$, one can also treat immersions $\Sigma \to (M,\mathcal{D})$ transverse to $\mathcal{D}$ as partially horizontal immersions \cite[pg. 256]{gromovCCMetric}. It turns out that all such maps are $\Omega_\bullet$-regular in the sense of Gromov. One can then get a stronger version of \autoref{thm:hPrinTransImm}, where the distribution can be taken to be arbitrary, and thus extending \autoref{thm:hPrinTransImmNonCritical}.
\end{remark}

\section*{Acknowledgments}
The author would like to thank M. Datta for many valuable comments and insightful discussions, and to S. Prasad for their help with the presentation of the article. The author would also like to thank the anonymous referee for pointing out a mistake in \autoref{lemma:specialVectorFields} which was originally stated without the equiregularity assumption, and for other suggestions to improve this article. This work was supported by IISER Kolkata funded Post-doctoral Fellowship (IISER-K PDF).


\end{document}